\documentclass[a4paper,12pt]{article}
\usepackage[top=3.1cm,bottom=2.6cm,left=2.6cm,right=2.6cm]{geometry}
\usepackage[tbtags]{amsmath}
\usepackage{amsmath, amssymb, amscd, amsthm, amsfonts}
\usepackage{geometry}
\usepackage{algorithm}
\usepackage[noend]{algpseudocode}
\usepackage{float}
\usepackage{caption} 
\usepackage{graphicx}
\usepackage{mathrsfs}
\usepackage{hyperref}
\usepackage{fancyhdr}
\usepackage{bm}
\usepackage{dsfont}
\usepackage{indentfirst}
\usepackage{url}
\usepackage{xcolor}
\newtheorem{theorem}{Theorem}[section]

\newtheorem{lemma}[theorem]{Lemma}
\newtheorem{fact}[theorem]{Fact}
\newtheorem{claim}[theorem]{Claim}
\newtheorem{proposition}[theorem]{Proposition}
\newtheorem{conjecture}[theorem]{Conjecture}

\newtheorem{question}[theorem]{Question}
\newtheorem{case}{Case}
\newtheorem{subcase}{Subcase}[case]

\newtheorem{definition}[theorem]{Definition}

\usepackage{enumerate}

\title{An El-Zahar Type Theorem in $3$-graphs under Codegree Condition}
\author{Yangyang Cheng\thanks{Mathematical Institute, University of Oxford, Oxford, UK. Email: \texttt{yangyang.cheng@maths.ox.ac.uk}.} 
\and Mengjiao Rao\thanks{Data Science Institute, Shandong University, Jinan, China. Email: \texttt{mengjiao\_rao@outlook.com}.}
\and Guanghui Wang\thanks{School of Mathematics, Shandong University, Jinan, China. Email: \texttt{ghwang@sdu.edu.cn}.}
\and Yuqi Zhao\thanks{School of Mathematics, Shandong University, Jinan, China. Email: \texttt{zhaoyuqi123@mail.sdu.edu.cn}.}
}
\date{}

\begin{document}
\maketitle

\begin{abstract}
A 3-uniform loose cycle, denoted by $C_t$, is a 3-graph on $t$ vertices whose vertices can be arranged cyclically so that each hyperedge consists of three consecutive vertices, and any two consecutive hyperedges share exactly one vertex. The length of $C_t$ is the number of its hyperedges. We prove that for any $\eta>0$, there exists an $n_0=n_0(\eta)$ such that for any $n\geq n_0$ the following holds. Let $\mathcal{C}$ be a 
$3$-graph consisting of vertex-disjoint loose cycles $C_{n_1}, C_{n_2}, \ldots, C_{n_r}$ such that $\sum_{i=1}^{r}n_i=n$. Let $k$ be the number of loose cycles with odd lengths in $\mathcal{C}$. If $\mathcal{H}$ is a 3-graph on $n$ vertices with minimum codegree at least $(n+2k)/4+\eta n$, then $\mathcal{H}$ contains $\mathcal{C}$ as a spanning subhypergraph. 
The degree condition is approximately tight.
This generalizes the result of K\"{u}hn and Osthus for loose Hamilton cycle and the result of Mycroft for loose cycle factors in 3-graphs. 
Our proof relies on the regularity lemma and a transversal blow-up lemma recently developed by the first author and Staden.
\end{abstract}


\textbf{Keywords}: loose cycles, codegree, regularity lemma, transversal blow-up lemma

\section{Introduction}\label{section-introduction}
In 1952, Dirac \cite{dirac1952some} proved that every $n$-vertex graph with a minimum degree of at least $n/2$ contains a Hamilton cycle. This result has motivated extensive research into determining the minimum degree required to ensure a spanning subgraph in a given graph $G$, i.e., a subgraph that has as many vertices as $G$.
Given a graph $G$ of order $n$ and a graph $F$ of order $h$, an $F$-\emph{tiling} of $G$ is a subgraph of $G$ that consists of vertex-disjoint copies of $F$. We say that $G$ has a \emph{perfect $F$-tiling} (or an \emph{$F$-factor}) if $h$ divides $n$ and $G$ contains $n/h$ vertex-disjoint copies of $F$. Hell and Kirkpatrick \cite{Hell1983sima}
showed that when the graph $F$ has any component with at least three vertices, the problem of determining whether $G$ has an $F$-factor is NP-complete.
This leads to a natural question: What sufficient conditions on $G$ guarantee an $F$-factor?

The celebrated Hajnal-Szemer\'{e}di \cite{HS1970CT} theorem states that every $n$-vertex graph with minimum degree at least $(k-1)n/k$ contains a $K_k$-factor. The case $k=3$ was earlier established by Corr\'{a}di and Hajnal \cite{CH1963acta}.
For any fixed graph $F$, Alon and Yuster \cite{AY1996} first determined the degree condition guaranteeing an $F$-factor, expressed in terms of the chromatic number of $F$. This was later refined by K\"{u}hn and Osthus \cite{KO2009com}, who established the exact threshold.
When it comes to tiling problems with different number of vertices, El-Zahar \cite{el1984circuits} proposed the following classical conjecture.
\begin{conjecture}[\cite{el1984circuits}, El-Zahar's conjecture]\label{ELZconj}
Let $H$ be a graph consisting of $r$ vertex-disjoint
cycles of order $n_1, n_2, \ldots, n_r$ satisfying $n_1+ n_2+\cdots+n_r=n$. 
If $G$ is a graph on $n$ vertices with minimum degree at least $\sum_{i=1}^{r}\lceil n_i/2\rceil$, then $G$ contains $H$ as a spanning subgraph.
\end{conjecture}

Abbasi \cite{Abbsi} confirmed this conjecture for sufficiently large $n$ using the classical regularity-blow-up method.
Note that the case $r=1$ reduces to Dirac's theorem.
The famous Corr\'{a}di-Hajnal theorem \cite{CH1963acta} corresponds to the special case where $n_1=n_2=\cdots=n_r=3$ (i.e., triangle factors). 
The case $n_1=n_2=\cdots=n_r=4$, conjectured by Erd\H{o}s and Faudree \cite{EF1990}, was solved by Wang \cite{WH2010}.

Over the past few decades, extending tiling problems to hypergraphs has attracted considerable attention. 
We first introduce some definitions.
Given an integer $k\geq2$, a \emph{$k$-uniform hypergraph} (or \emph{$k$-graph} for short) consists of a vertex set $V$ and an edge set $E\subseteq\binom{V}{k}$.
In this paper, we only focus on 3-graphs. 
A hyperedge $\{x, y, z\}$ is denoted by $(xyz)$ for simplicity. 
For a $3$-graph $\mathcal{H}=(V, E)$ and two distinct vertices $u, v\in V$, let $N_\mathcal{H}(v):=\{\{x, y\}\in\binom{V}{2}:(vxy)\in E(\mathcal{H})\}$ and $N_\mathcal{H}(u, v):=\{w\in V:(uvw)\in E(\mathcal{H})\}$.
The \emph{degree} of a vertex $v\in V$, denoted by $d_{\mathcal{H}}(v)$, is defined to be $|N_\mathcal{H}(v)|$.
The \emph{minimum vertex degree} of $\mathcal{H}$ is defined as $\delta_1(\mathcal{H}):=\min_{v\in V}|N_{\mathcal{H}}(v)|$, and the \emph{minimum codegree} of $\mathcal{H}$ is defined as $\delta_2(\mathcal{H}):=\min_{u,v\in\binom{V}{2}}
|N_{\mathcal{H}}(u, v)|$.
A \emph{loose cycle} $C_t$ is a 3-graph on $t$ vertices whose vertices can be ordered cyclically such that hyperedges are triples of consecutive vertices, and consecutive hyperedges overlap in exactly one vertex.
Note that this definition indicates that all the loose cycles we focus on in this paper have $t\geq 6$.

As a natural generalization of Dirac's theorem, K\"{u}hn and Osthus \cite{kuhn2006loose} gave an approximate bound that ensures loose Hamilton cycles in $3$-graphs.

\begin{theorem}[\cite{kuhn2006loose}]\label{ko}
For each $\sigma>0$, there is an integer $n_0=n_0(\sigma)$ such that every $3$-graph on $n\geq n_0$ vertices with minimum codegree at least $(1/4+\sigma)n$ contains a loose Hamilton cycle.
\end{theorem}

Czygrinow and Molla \cite{czygrinow2014siam} proved that a minimum codegree of $n/4$ suffices for the loose Hamilton cycle.
In \cite{kuhn2006loose}, K\"{u}hn and Osthus further showed that the codegree condition in Theorem \ref{ko} guarantees a $C_4$-factor, while Czygrinow, DeBiasio and Nagle \cite{CDN2014JGT} determined the exact codegree threshold.
Gao and Han \cite{Hanjie2017CPC} determined that the tight minimum codegree forcing a $C_6$-factor is $n/3$.
Answering a question of R\"{o}dl and Ruci\'{n}ski (\cite{2010Rodl}, Problem 3.15), Czygrinow \cite{Czy2016JGT} established the tight codegree threshold for $C_s$-factors in $3$-graphs for all even integer $s \geq 6$.

\begin{theorem}[\cite{Czy2016JGT}]\label{MYC}
For every even integer $s \geq 6$, there is $n_0$ such that every $3$-graph $\mathcal{H}$ on $n \in s \mathbb{Z}, n \geq n_0$ vertices with $\delta_2(\mathcal{H}) \geq \lceil s / 4\rceil n/s$ has a $C_s$-factor.
\end{theorem}

Mycroft \cite{Myc2016JCTA} generalized the above result to $k$-graphs and gave an asymptotic minimum codegree condition for such factors. 
In this paper, we consider an analogue of El-Zahar's conjecture in $3$-graphs as follows.

\begin{theorem}[Main result]\label{oddandeven}
For every $\eta>0$, there exists an integer $n_0=n_0(\eta)$ such that the following holds for any $n\geq n_0$. Let $\mathcal{C}$ be a $3$-graph consisting of vertex-disjoint loose cycles $C_{n_1}, C_{n_2}, \ldots, C_{n_r}$, where $\sum_{i=1}^{r}n_i=n$. Let $k$ denote the number of loose cycles with odd lengths. If $\mathcal{H}$ is an $n$-vertex $3$-graph with $\delta_2(\mathcal{H})\geq(n+2k)/4+\eta n$, then $\mathcal{H}$ contains $\mathcal{C}$ as a spanning subhypergraph.
\end{theorem}

Note that the result of K\"{u}hn and Osthus \cite{kuhn2006loose} for loose Hamilton cycles corresponds to the special case $r=1$. The bound in Theorem \ref{MYC} is obtained by taking $n_1=n_2=\cdots=n_r=n/s$, with an error term of $o(n)$.
The bound in Theorem \ref{oddandeven} is asymptotically tight, which follows from the extremal construction in the following proposition.

\begin{proposition} \label{extremal}
Let $n,k$ be positive integers. Suppose that $\mathcal{C}$ is a 3-graph consisting of vertex-disjoint loose cycles $C_{n_1}, C_{n_2}, \ldots, C_{n_r}$ with $\sum_{i=1}^{r}n_i=n$, and let $k$ be the number of loose cycles with odd lengths. There exists a 3-graph $\mathcal{H}$ with $n$ vertices and minimum codegree $\lfloor\frac{n+2k}{4}\rfloor-1$ that does not contain $\mathcal{C}$ as a spanning subhypergraph. \end{proposition}

\begin{proof} 
We construct the 3-graph $\mathcal{H}$ as follows.
The vertex set of $\mathcal{H}$ is the disjoint union of two sets $A$ and $B$, where $|A|=\lfloor\frac{n+2k}{4}\rfloor-1$ and $|B|=n-|A|$. 
$\mathcal{H}$ contains exactly those hyperedges whose intersection with $A$ is nonempty. 
For any cycle $C_{n_i}$, if it is an even cycle, then its vertex cover number is $n_i/4$, and if it is an odd cycle, then its vertex cover number is $(n_i+2)/4$. Thus, the vertex cover number of $\mathcal{C}$ is $(n+2k)/4$. 
This leads to a contradiction since $A$ is a vertex cover of $\mathcal{H}$. Consequently, $\mathcal{H}$ does not contain $\mathcal{C}$ as a spanning subhypergraph.
\end{proof}

This paper is organized as follows.
We end this section with a proof sketch of the main theorem.
In Section \ref{ES}, we give the necessary definitions and the probabilistic tool.
Section \ref{section-preliminaries} collects the hypergraph regularity method, and the transversal blow-up lemma used in our proof.
In section \ref{section-proof}, we complete the proof of Theorem \ref{oddandeven}.
Section \ref{concluding} contains some open problems.

\textbf{Proof sketch of the main theorem.}\label{sketch}
The proof of Theorem \ref{oddandeven} uses the weak hypergraph regularity method and the blow-up lemma for transversals. We first reduce the problem to the case where the length of each loose cycle is at most $O (\log n)$ through the proof in Appendix \ref{appendix}, and divide the remaining proof into the following three steps.

\textbf{Step 1:} Apply the weak hypergraph regularity lemma to obtain the reduced hypergraph and find an almost perfect tiling of some defined $3$-graph in the reduced hypergraph. Next, we recover each component of this tiling back to $\mathcal{H}$, and perform a random slicing to obtain some regular triples $T(X, Y, Z)$ that cover almost all of the vertex set, where the sizes of $X$, $Y$, and $Z$ satisfy a specific proportional relationship. Using the slicing lemma, we can ensure that these triples are half-superregular.

\textbf{Step 2:} Embed some cycles from the list to use up the vertices not contained in any of the triples. Embed more cycles in the list to consume some of the vertices in the triples such that each tripe has only been consumed a little fraction of vertices, and the remaining cycles to be embedded could be partitioned into pieces that allows us to build an injection from the piece set to the triple set.

\textbf{Step 3:} For each triple, apply the transversal blow-up lemma to embed the set of cycles from the corresponding piece.

\section{Preliminaries}\label{ES}
\subsection{Definitions and notations}
We simply refer to “graphs” when we consider simple, undirected and finite 2-graphs.
Given a graph $G$,
a 1\emph{-expansion} $G^+$ of $G$ is a $3$-graph obtained by replacing every edge $e=\{x,y\}\in E(G)$ with a hyperedge $(xyt_{e})$, where all the vertices $t_e (e\in E(G))$ are distinct and they are disjoint from $V(G)$.
We call these added vertices $t_e$ $(e\in E(G))$ \emph{new vertices} of the 3-graph $G^+$. 
Note that a \emph{loose cycle} is a $1$-expansion of a cycle and a \emph{loose path} is a $1$-expansion of a path.
We use $P_t$ to denote a loose path on $t$ vertices. 
Once the order of $P_t$ is determined, we call a vertex that lies in its first hyperedge but not in the second one the \emph{endpoints} of $P_t$. 
The \emph{length} of $P_t$ (or $C_t$), denoted by $\ell(P_t)$ (or $\ell(C_t)$), is the number of its hyperedges. 
A 3-graph on vertex sets $V_1, V_2$ and $V_3$ is called a 3-\emph{partite} 3-\emph{graph} if each hyperedge contains exactly one vertices of $V_i$ for each $i\in[3]$.
Let $\mathcal{H}$ be a $3$-graph, we use $|\mathcal{H}|$ to denote the order of $\mathcal{H}$ and $e(\mathcal{H})$ to denote the number of hyperedges of $\mathcal{H}$. 
For $S\subseteq V(\mathcal{H})$, we use $\mathcal{H}[S]$ to denote the subgraph of $\mathcal{H}$ induced by $S$.
For $u, v\in V(\mathcal{H})$, let $N_{S}(u):=\{\{v,w\}\in \binom{S}{2}: (uvw)\in E(\mathcal{H})\}$ and $N_{S}(u,v):=\{w\in S: (uvw)\in E(\mathcal{H})\}$.


Throughout this paper, we use $\log$ to denote the logarithm with base 2. 
For a set $X$, we set $\binom{X}{k}:=\{A\subset X: |A|=k\}$. 
In addition, we write $[t]:=\{1, 2, ...,t\}$ for any integer $t\geq1$.
For any two constants $a,b \in (0, 1)$, we write $a \ll b$ if there exists a function $a_0 = a_0(b)$ such that the subsequent arguments hold for all $0 < a \leq a_0$.
In order to simplify the presentation, we omit floors and ceilings and treat large numbers as integers whenever this does not affect the argument.

For two $k$-graphs $\mathcal{H}_1$ and $\mathcal{H}_2$, an \emph{embedding} from $\mathcal{H}_1$ into $\mathcal{H}_2$ is an injective edge-preserving mapping $\psi: V( \mathcal{H}_1)\to V(\mathcal{H}_2)$. Given a such mapping $\psi$, we write $\psi (X):=\{\psi (x): x\in X\}$ for $X\subseteq V(\mathcal{H}_1)$. Moreover, if $|\mathcal{H}_1|=|\mathcal{H}_2|$, we say the embedding is a \emph{perfect embedding}.

Now we present a key definition that will be essential for the proof.

\begin{definition}
Let $n, k$ be positive integers with $k<n$, and let $\eta$ be a constant satisfying $0<\eta<1$. 
Two positive integers $a, b$ are called $(n, k, \eta)$-good if the following holds.
\begin{enumerate}[$(i)$]
    \item $\frac{2n-4k-0.08\eta n}{n+2k+0.04\eta n}\leq\frac{b}{a}\leq \frac{2n-4k}{n+2k}-\eta^2$;
    \item $a,b\leq \lceil\frac{200}{\eta}\rceil$.
\end{enumerate}
\end{definition}

\begin{fact}\label{goodrange}
Let $n, k$ be two positive integers with $k<n$ and $\eta$ be a constant with $1/n\ll\eta\ll 1$. For $a=\lceil 100/\eta \rceil$, there exists a positive integer $b$ with $1\leq b\leq \lceil 200/\eta \rceil$ such that $a,b$ are $(n, k, \eta)$-good.
\end{fact}

\begin{proof}
Note that 
\[ 
\frac{2n - 4k}{n + 2k} - \eta^2 - \frac{2n - 4k -0.0 8\eta n}{n + 2k +0.0 4\eta n}\geq \frac{0.08\eta n}{n + 2k +0.0 4\eta n} - \eta^2\geq \frac{0.08\eta}{3 + 0.04\eta} - \eta^2 \geq 0.02\eta.
\]

Since $\frac{2n-4k}{n+2k}-\eta^2-\frac{2n-4k-0.08\eta n}{n+2k+0.04\eta n}\geq 0.02\eta$, there exists a positive integer $b>0$ such that $b/a\in [\frac{2n-4k-0.08\eta n}{n+2k+0.04\eta n}, \frac{2n-4k}{n+2k}-\eta^2].$ Moreover, $1\leq b\leq a(\frac{2n-4k}{n+2k}-\eta^2)\leq \lceil\frac{200}{\eta}\rceil$.
\end{proof}

\subsection{Basic probabilistic estimates}
In this subsection, we will introduce a probabilistic tool given by K\"{u}hn and Osthus \cite{kuhn2006loose}.
Given a positive number $\gamma$, set $T$ and sets $A,Q\subseteq T$, we say that $A$ is \emph{split $\gamma$-fairly} by $Q$ if:
$$
|\frac{|A\cap Q|}{|Q|}-\frac{|A|}{|T|}|\leqslant\gamma.
$$

Thus, if $\gamma$ is small and $A$ is split $\gamma$-fairly by $Q$, then the proportion of all those elements of $T$ which lie in $A$ is almost equal to the proportion of all those elements of $Q$ which lie in $A$. The following fact implies that if $Q$ is random, it tends to split large sets $\gamma$-fairly.

\begin{lemma}[\cite{kuhn2006loose}, Proposition 4.1]\label{fairly}
For each $0<\gamma<1$, there exists an integer $q_0=q_0(\gamma)$ such that the following holds. Given $t\geqslant q\geqslant q_0$ and a set $T$ of size $t$, let $Q$ be a subset of $T$ obtained by successively selecting $q$ elements uniformly at random without repetitions. Let $\mathcal{A}$ be a family of at most $q^{10}$ subsets of $T$ such that $|A|\geqslant\gamma t$ for each $A\in\mathcal{A}$. Then, with probability at least $1/2$, every set in $\mathcal{A}$ is split $\gamma$-fairly by $Q$.
\end{lemma}

\setcounter{case}{0}
\subsection{Embedding cycles in tripartite graph}\label{graphcase} 
In this section we give a proof regarding embedding cycles in tripartite graph.
\begin{lemma}\label{colorcycle}
Let $n$ be a positive integer. For any cycle $C$ of length $n$ and any integers $a,b,c$ satisfying $a+b+c=n$ and $0\leq a\leq b\leq c\leq\lfloor n/2\rfloor$, there exists a proper $(a,b,c)$-coloring of $C$, i.e., $C$ has a proper vertex coloring such that the three monochromatic parts have sizes $a,b$ and $c$.
\end{lemma}
\begin{proof}
Without loss of generality, suppose that the vertex set of $C$ is $\mathbb{Z}_n=\mathbb{Z}/n\mathbb{Z}$ and the edge set of $C$ is $\{(i,i+1) \mid i\in \mathbb{Z}_n\}$. We partition the vertex set $V(C)$ into two parts $A$ and $B$, where $A=\{i\in \mathbb{Z}_n \mid 1\leq i\leq n, i\ \text{is odd}\}$ and $B=\{i\in \mathbb{Z}_n \mid 1\leq i\leq n, i\ \text{is even}\}$. 
We will color the vertices of $C$ with red, blue and green as follows. 
Note that $|A|=\lceil n/2\rceil$ and $c\leq\lfloor n/2\rfloor$, thus we have $\{1,3, \ldots, 2c-1\}\subseteq A$ and these vertices are colored red. 
Let $A'=A\backslash \{1,3, \ldots, 2c-1\}$. 
If $|A'|\geq b$, then we have $|B|=n-|A|=n-c-|A'|\leq n-c-b=a$.
Since $|B|=\lfloor n/2\rfloor$ and $a\leq\lfloor n/2\rfloor$, we know that $a=\lfloor n/2 \rfloor$ and $|A'|=b$. 
We color all vertices in $A'$ blue and color all vertices in $B$ green. 
It is easy to check that this vertex coloring is proper.

In the following, we may suppose that $|A'|<b$. 
We color all vertices in $A'$ and the set $\{2,4,\ldots,2(b-|A'|)\}$ blue, and color the rest of the vertices in $B$ green. 
We claim that this vertex coloring of $C$ is proper.
It's easy to see that adjacent vertices colored the same color could only exist in the blue vertex set.  This implies that $2(b-|A'|)\geq 2c$. 
Recall that $|A'|=\lceil n/2 \rceil-c$, we obtain $b\geq \lceil n/2 \rceil$. Thus, we have $b=c=n/2$ and $n$ is even by our assumption. 
If $A'$ is empty, then $n$ is even and $c=n/2$.
We color the vertices in $A$ red and color the set $\{2,4,\ldots,2b\}$ blue. The rest of the vertices $B$ are colored green. Clearly, this vertex coloring is proper, as desired.
\end{proof}

\begin{lemma}\label{step3graph}
Let $n_1, n_2, \ldots, n_r$ and $k$ be positive integers. Let $F$ be a graph consisting of $r$ vertex-disjoint cycles of lengths $n_1, n_2, \ldots, n_r$ with $n_1+n_2+\cdots+n_r=n$ and the number of odd cycles in $F$ be $k$.
Suppose that $H$ is an $n$-vertex complete tripartite graph on vertex sets $V_1, V_2$ and $V_3$. 
If $H$ has minimum degree at least $\sum_{i=1}^{r}\lceil n_i/2\rceil=n/2+k/2$, then $H$ contains $F$ as a subgraph.
\end{lemma}
\begin{proof}
\setcounter{case}{0}

Without loss of generality, assume that $|V_1| \leq |V_2| \leq |V_3|$. Note that $n - k$ is even since $n$ and $k$ have the same parity. Owing to the minimum degree condition of $H$ and the fact that $n\geq3k$, we have $k\leq(n+k)/4\leq |V_1| \leq |V_2| \leq |V_3| \leq (n - k)/2$. Besides, $k\leq |V_1| \leq |V_2| \leq |V_3|\leq (n-k)/2$ implies the minimum degree is no less than $n/2+k/2$. Threrfore, the minimum degree condition is equivalent to requiring $k\leq |V_1| \leq |V_2| \leq |V_3|\leq (n-k)/2$.
We proceed by induction on the triple $(|V_1|, |V_2|, |V_3|)$ under lexicographic order. 

\noindent\textbf{Base case}: Suppose $|V_1| = k$ and $|V_2| = |V_3| = (n-k)/2$. For any cycle $C \in F$, by Lemma~\ref{colorcycle}, $C$ admits a $(0, \lfloor |C|/2 \rfloor, \lfloor |C|/2 \rfloor)$-coloring if $|C|$ is even, and a $(1, \lfloor |C|/2 \rfloor, \lfloor |C|/2 \rfloor)$-coloring if $|C|$ is odd. Denote the three parts in the coloring by $D_1, D_2, D_3$, and embed them into $V_1, V_2, V_3$, respectively. Applying this embedding for each $C \in F$ yields an embedding of $F$ into $H$.
\begin{claim}
Let $e_{12} = (1, -1, 0)$, $e_{13} = (1, 0, -1)$, and $e_{23} = (0, 1, -1)$. For any integers $a, b, c$, define $(a(ij), b(ij), c(ij)) = (a, b, c) + e_{ij}$. Suppose that $H$ contains $F$ as a subgraph when $(|V_1|, |V_2|, |V_3|) = (a, b, c)$. Then, for any $i, j \in [3]$ with $i < j$, if $(a(ij), b(ij), c(ij))$ satisfies $a(ij) \leq b(ij) \leq c(ij)$, it follows that $H$ also contains $F$ as a subgraph when $(|V_1|, |V_2|, |V_3|) = (a(ij), b(ij), c(ij))$.
\end{claim}
\begin{proof}
Since $a(ij) \leq b(ij) \leq c(ij)$, when $(|V_1|, |V_2|, |V_3|) = (a, b, c)$, it follows that $|V_j| - |V_i| \geq 2$. Thus, when we embed $F$ into $H$ in this case and color all vertices in $V_i$ with color $c_i$ for all $i \in [3]$, one of the two cases holds:
\begin{case}
There exists at least one cycle $C$ such that the number of its vertices colored by $c_j$ is at least two more than those colored by $c_i$.
\end{case}
Denote by $x_k$ the number of vertices in $C$ colored with color $c_k$ for each $k \in [3]$. Then we have $x_j \geq x_i + 2$. Let $h$ be an integer satisfying $\{h\} = \{1, 2, 3\} \setminus \{i, j\}$. Since $(x_i, x_j, x_h)$ is a proper coloring of $C$, by Lemma~\ref{colorcycle}, $(x_i + 1, x_j - 1, x_h)$ is also a proper coloring of $C$. Denote the three parts in the ordering by $D_i, D_j, D_h$. Thus, when $(|V_1|, |V_2|, |V_3|) = (a(ij), b(ij), c(ij))$, we can embed $F$ into $H$ by embedding $C$ into $H$ via $D_i$, $D_j$, and $D_h$ into $V_i$, $V_j$, and $V_h$ respectively, and embedding the remaining cycles in the same way as when $(|V_1|, |V_2|, |V_3|) = (a, b, c)$.

\begin{case}
There exist at least two cycles, each of which has at least one more vertex colored by $c_j$ than by $c_i$.
\end{case}
Denote these cycles by $C_1$ and $C_2$. For each $l \in [2]$ let $x^l_k$ denote the number of vertices in $C_l$ colored with color $c_k$ for each $k \in [3]$, and observe that $x^l_j \geq x^l_i + 1$. Let $h$ be an integer satisfying $\{h\} = \{1, 2, 3\} \setminus \{i, j\}$. Since for each $l\in [2]$, $(x^l_i, x^l_j, x^l_h)$ is a proper coloring of $C_l$, by Lemma~\ref{colorcycle}, $(x^l_i + 1, x^l_j - 1, x^l_h)$ is also a proper coloring of $C_l$. Denote the corresponding parts by $D^l_i, D^l_j, D^l_h$ for each $l \in [2]$. Then, when $(|V_1|, |V_2|, |V_3|) = (a(ij), b(ij), c(ij))$, we can embed $F$ into $H$ by embedding each $C_l$ into $H$ via $D^l_i$, $D^l_j$, and $D^l_h$ into $V_i$, $V_j$, and $V_h$ respectively, and embedding the remaining cycles as in the case of $(a, b, c)$.
\end{proof}
Thus as long as $k\leq |V_1| \leq |V_2| \leq |V_3|\leq (n-k)/2$, $H$ contains $F$ as a subgraph.
\end{proof}
\setcounter{case}{0}

\section{The regularity method and blow-up lemma for 3-graphs}\label{section-preliminaries}
The well-known weak hypergraph regularity lemma is a straightforward extension of Szemer{\'e}di's regularity lemma for graphs \cite{szemeredi1975regular}. 
Prior to this paper, the strong hypergraph regularity lemma was commonly used in the spanning subgraph embedding problem for hypergraphs, see \cite{2011Keevash, kuhn2006loose, 2011Rodl, 2006Rodl} and the references. 
In this paper, we only need to use the weak hypergraph regularity lemma. 

\subsection{Regularity method for 3-graphs}

Let $\mathcal{H}$ be a $3$-partite $3$-graph with disjoint vertex sets $V_1, V_2$ and $V_3$, which we also denote as $(V_1, V_2, V_3)_{\mathcal{H}}$. Each hyperedge in $\mathcal{H}$ contains exactly one vertex from each set $V_i$ for $i \in [3]$. Let $e(\mathcal{H})$ denote the number of edges in $\mathcal{H}$. The \emph{density} of $\mathcal{H}$ is defined as:
$$
d_{\mathcal{H}}(V_1, V_2, V_3)=\frac{e(\mathcal{H})}{|V_1||V_2||V_3|}.
$$
Given $\varepsilon>0$, we say that $(V_1, V_2, V_3)_{\mathcal{H}}$ is

\begin{itemize}
\item (weakly) \emph{$\varepsilon$-regular} if for every subhypergraph $(V_1', V_2', V_3')_{\mathcal{H}}$ with $V_i' \subseteq V_i$ and $|V_i'|\geq\varepsilon|V_i|$ for each $i\in[3]$, we have
$$
|d_{\mathcal{H}}(V_1', V_2', V_3')-d_{\mathcal{H}}(V_1, V_2, V_3)|<\varepsilon;
$$
\item (weakly) \emph{$(\varepsilon, d)$-regular} if additionally $d_{\mathcal{H}}(V_1, V_2, V_3)\geq d$;
\item (weakly) \emph{$(\varepsilon, d)$-half-regular} if for every subhypergraph $(V_1', V_2', V_3')_{\mathcal{H}}$ with $V_i' \subseteq V_i$ and $|V_i'|\geq \varepsilon|V_i|$ for each $i\in[3],$ we have $d_{\mathcal{H}}(V_1', V_2', V_3')\geq d$;
\item (weakly) \emph{$(\varepsilon, d)$-superregular} if it is $(\varepsilon, d)$-regular and additionally,
    $$d_{\mathcal{H}}(x)\geq \frac{d|V_1||V_2||V_3|}{|V_i|}$$ 
    for each $i\in[3]$ and $x\in V_i$;
\item (weakly) \emph{$(\varepsilon, d)$-half-superregular} if for every subhypergraph $(V_1', V_2', V_3')_{\mathcal{H}}$ with $V_i' \subseteq V_i$ and $|V_i'|\geq \varepsilon|V_i|$ for each $i\in[3],$ we have $d_{\mathcal{H}}(V_1', V_2', V_3')\geq d$ and $$d_{\mathcal{H}}(x)\geq \frac{d|V_1||V_2||V_3|}{|V_i|}$$ for each $i\in[3]$ and $x\in V_i$.
\end{itemize}
We will generally omit the terms weakly when using these definitions in the following.

The following so-called slicing lemma was introduced by the first author and Staden, see Lemma 2.4 in \cite{cheng2023transversals}.  
Note that we use the term ``half-(super)regular", which can be obtained by the original proof with some additional trivial analysis. 

\begin{lemma}[\cite{cheng2023transversals}, Slicing lemma]\label{slicinglemma}
Let $0<1/m\ll\varepsilon\ll\alpha\ll d\leq 1$. Let $\mathcal{H}$ be a $3$-partite $3$-graph with vertex sets $V_1, V_2, V_3$ and $m\leq \min\{|V_1|, |V_2|, |V_3|\}$. Let $V'_i\subseteq V_i$ for each $i\in [3]$ and $\mathcal{H'}$ be the subhypergraph of $\mathcal{H}$ induced on $V_1'\cup V_2'\cup V_3'$. Then the following holds.
\begin{enumerate}[$(i)$]
\item Suppose that $\mathcal{H}$ is $(\varepsilon, d)$-$(half)$-regular, and $|V'_i|\geq \alpha|V_i|$ for each $i\in [3]$. Then $\mathcal{H'}$ is $(\varepsilon/\alpha, d/2)$-$(half)$-regular.
\item Suppose that $\mathcal{H}$ is $(\varepsilon, d)$-$(half)$-superregular, and $|V'_i|\geq (1-\alpha)|V_i|$ for each $i\in [3]$. Then $\mathcal{H'}$ is $(2\varepsilon, d/2)$-$(half)$-superregular.
\item Suppose that $\mathcal{H}$ is $(\varepsilon, d)$-$(half)$-superregular, and $m_i\geq \alpha|V_i|$ for each $i\in [3]$. If $V'_i\subseteq V_i$ is a uniform random subset of size $m_i$ for each $i\in[3]$, then with high probability, $\mathcal{H'}$ is $(\varepsilon/\alpha, d^2/16)$-$(half)$-superregular.
\end{enumerate}
\end{lemma}

Using the slicing lemma, it is easy to obtain the following facts. It shows that we can extract a superregular $3$-partite subhypergraph from a given regular $3$-partite $3$-graph by removing a small number of vertices from each vertex set. 

\begin{lemma}\label{regtosupreg}
Let $0<\varepsilon\ll d\leq1$, and let $\mathcal{H}$ be an $(\varepsilon, d)$-$(half)$-regular $3$-partite $3$-graph with vertex sets $V_1, V_2$, and $V_3$. Then there exist subsets $V_i'\subseteq V_i$ with $|V_i'|=(1-\varepsilon)|V_i|$ for each $i\in[3]$ such that $(V_1', V_2', V_3')_{\mathcal{H}}$ is $(2\varepsilon, d/2)$-$(half)$-superregular.
\end{lemma}
\begin{proof}
Let $V_i'':=\{v\in V_i: d_{\mathcal{H}}(v)<(d-3\varepsilon)|V_j||V_h|\}$ for $\{j, h\}=[3]\setminus\{i\}$ and each $i\in[3]$. We claim that $|V_i''|<\varepsilon|V_i|$ for each $i\in[3]$. Indeed, suppose for contradiction that $|V_1''|\geq\varepsilon|V_1|$. Then we have $d_{\mathcal{H}}(V_1'', V_2, V_3)>d-\varepsilon$. However, by the definition of $V_1''$, we have $e_{\mathcal{H}}(V_1'', V_2, V_3)<(d-3\varepsilon)|V_1''||V_2||V_3|$. This implies that $d_{\mathcal{H}}(V_1'', V_2, V_3)<d-3\varepsilon$, a contradiction.

Let $V_i'\subseteq V_i\setminus V_i''$ such that $|V_i'|=(1-\varepsilon)|V_i|$ for each $i\in [3]$. We show that $\mathcal{H}'=(V_1', V_2', V_3')_{\mathcal{H}}$ is $(2\varepsilon, d-3\varepsilon)$-(half)-superregular. By Lemma \ref{slicinglemma} (i), $\mathcal{H}'$ is $(2\varepsilon, d/2)$-(half)-regular.
Moreover, for each vertex $v\in V_i'$ and $\{j, h\}=[3]\setminus\{i\}$, we have
$$d_{\mathcal{H}'}(v)\geq(d-3\varepsilon)|V_j||V_h|-2\varepsilon |V_i||V_j|\geq d|V'_j||V'_h|/2.$$
Therefore, $\mathcal{H'}$ is $(2\varepsilon, d/2)$-(half)-superregular.
\end{proof}

Now, we present the weak hypergraph regularity lemma, which follows from Chung \cite{chung1991RSA}.

\begin{lemma}[Weak hypergraph regularity lemma]\label{WeakRegLem}
For all integers $L_0\geq 1$ and for every $\varepsilon>0$, there exists $T_0=T_0(t_0, \varepsilon)$ such that for sufficiently large $n$ and every $3$-graph $\mathcal{H}=(V, E)$ on $n$ vertices, there exists a partition $V=V_0\cup V_1\cup\cdots \cup V_L$ satisfying:
\begin{enumerate}[$(i)$]
\item $L_0 \leq L \leq T_0$;
\item $\left|V_1\right|=\cdots=\left|V_L\right|$ and $\left|V_0\right| \leq\varepsilon n$;
\item for all but at most $\varepsilon\binom{L}{3}$ triples $\left\{i, j, k\right\}\subseteq[L]$, we have that $(V_{i}, V_{j}, V_{k})_{\mathcal{H}}$ is $\varepsilon$-regular.
\end{enumerate}
\end{lemma}

The vertex partition in Lemma \ref{WeakRegLem} is called an \emph{$\varepsilon$-regular partition} of $\mathcal{H}$.
We call $V_0$ the \emph{exceptional set}, and $V_1, \ldots, V_L$ \emph{clusters}.
For our purpose, we will use the degree form of the weak hypergraph regularity lemma \cite{kuhn2014fractional}, which is proved in the same way as the original regularity lemma of graphs (see a proof in \cite{townsend2016extremal}).
\begin{lemma}[Degree form of the weak hypergraph regularity lemma]\label{degweakreg}
For any integer $L_0 \geq 1$ and every $\varepsilon>0$, there is an integer $n_0=n_0(\varepsilon, L_0)$ such that for every $d\in[0,1)$ and for every $3$-graph $\mathcal{H}=(V, E)$ on $n \geq n_0$ vertices, there exists a partition of $V$ into $V_0, V_1, \ldots, V_L$ and a spanning subhypergraph $\mathcal{H'}$ of $\mathcal{H}$ such that the following properties hold:
\begin{enumerate}[$(i)$]
\item $L_0 \leq L \leq n_0$ and $\left|V_0\right| \leq \varepsilon n$;
\item $\left|V_1\right|=\cdots=\left|V_L\right|:=m$;
\item $d_{\mathcal{H'}}(v)>d_{\mathcal{H}}(v)-(d+\varepsilon) n^2$ for each $v \in V$;
\item every edge of $\mathcal{H'}$ with more than one vertex in a single cluster $V_i$ for some $i \in[L]$ has at least one vertex in $V_0$;
\item for all triples $\{i, j, k\} \in\binom{L}{3}$, we have that $\left(V_i, V_j, V_k\right)_{\mathcal{H'}}$ is either empty or $(\varepsilon, d)$-regular.
\end{enumerate}
\end{lemma}

When applying the regularity lemma, a crucial point is to consider the reduced hypergraph, defined precisely in Definition \ref{reducegraph}.

\begin{definition}[Reduced hypergraph]\label{reducegraph}
Let $\mathcal{H}=(V, E)$ be a $3$-graph. Given parameters $\varepsilon>0$, $d \in[0,1)$ and $L_0\geqslant1$, we define the reduced hypergraph $\mathcal{R}=\mathcal{R}\left(\varepsilon, d, L_0\right)$ of $\mathcal{H}$ as follows. Apply the degree form of the weak hypergraph regularity lemma to $\mathcal{H}$, with parameters $\varepsilon, d, L_0$ to obtain a spanning subhypergraph $\mathcal{H'}$ and a partition $V_0, \ldots, V_L$ of $V$. Then $\mathcal{R}$ has vertices $V_1, \ldots, V_L$, and there exists an edge between $V_{i}, V_{j}, V_{k}$ precisely when $\left(V_{i}, V_{j}, V_{k}\right)_{\mathcal{H}'}$ is $(\varepsilon, d)$-regular.
\end{definition}

The reduced hypergraph $\mathcal{R}$ is also a $3$-graph. The following lemma is the case $k=3$ and $\ell=2$ of Lemma 5.5 in \cite{kuhn2014fractional}, which shows that the reduced hypergraph almost inherits the vertex degree of the original hypergraph.
\begin{lemma}[\cite{kuhn2014fractional}, Lemma 5.5]\label{degreeinherit}
Suppose $c>0, L_0\geq1$ and $0<\varepsilon\leq d \leq c^3/64$. Let $\mathcal{H}$ be an $n$-vertex $3$-graph with $\delta_{2}(\mathcal{H})\geq cn$. Suppose $\mathcal{R}=\mathcal{R}(\varepsilon, d, L_0)$ be the reduced hypergraph of $\mathcal{H}$. Then at least $\binom{|V(\mathcal{R})|}{2}-36d^{1/3}|\mathcal{R}|^{2}$ of the $2$-tuples of vertices of $\mathcal{R}$ have degree at least $(c-4d^{1/3})|\mathcal{R}|$.
\end{lemma}

\subsection{Weak 3-graph blow-up lemma}
In this subsection, we present the transversal blow-up lemma developed recently by the first author and Staden (Theorem 1.8 in \cite{cheng2023transversals}). Specifically, we will apply the $3$-graph version as follows (Theorem 6.1 in \cite{cheng2023transversals}) to derive a key result for our proof, which we call the path embedding lemma.

An $n$-vertex 2-graph $G$ is \emph{$\mu$-separable} if there is $X\subseteq V(G)$ of size at most $\mu n$ such that the induced subgraph $G[V(G)\setminus X]$ consists of components of size at most $\mu n$. For instance, trees, 2-regular graphs, powers of Hamilton cycles, and graphs with small bandwidths are separable graphs. 

\begin{theorem}[\cite{cheng2023transversals}, Weak $3$-graph blow-up lemma]\label{weakblowup}
Let $0<1/m\ll\varepsilon, \mu, \alpha\ll\nu, d, \delta, 1/\Delta, 1/r\leq1$.
\begin{enumerate}[$(i)$]
\item Let $R$ be a $2$-graph with vertex set $[r]$.
\item Let $\mathcal{H}$ be a $3$-graph with vertex sets $V_1, \ldots, V_r$ and $V_{i j}$ for $ij\in E(R)$ where $m \leq|V_i|\leq m/\delta$ for all $i\in[r]$, and $|V_{ij}|\geq \delta m$ for all $ij\in[r]$. Suppose that $\mathcal{H}[V_i, V_j, V_{i j}]$ is $(\varepsilon, d)$-$(half)$-superregular for all $ij\in E(R)$.
\item Let $G$ be a $\mu$-separable $2$-graph with $\Delta(G)\leq\Delta$ for which there is a graph homomorphism $\phi: V(G)\rightarrow V(R)$ with $|\phi^{-1}(i)|=|V_i|$ for all $i\in[r]$ and $e(G[\phi^{-1}(i), \phi^{-1}(j)])=|V_{i j}|$ for all $ij\in E(R)$.
\item Suppose that for each $i \in[r]$ there is a set $U_i \subseteq \phi^{-1}(i)$ with $|U_i|\leq\alpha m$ and $T_x \subseteq V_{\phi(x)}$ with $|T_x|\geq\nu m$ for all $x\in U_i$.
\end{enumerate}
Then $\mathcal{H}$ contains a copy of $G^+$, where each vertex $x\in V(G)$ is mapped to $V_{\phi(x)}$ and the new vertex $t_{xy} (xy \in E(G))$ is mapped to $V_{\phi(x) \phi(y)}$; and moreover, for every $i\in[r]$ every $x\in U_i$ is mapped to $T_x$.
\end{theorem}

Using Lemma \ref{weakblowup}, we prove the following path embedding lemma.

\begin{lemma}[Path embedding lemma]\label{pathembed}
Let $m$ be a integer and $0<1/m\ll\varepsilon\ll d\leq 1$. 
Suppose $G$ is a \((\varepsilon, d)\)-$(half)$-superregular $3$-partite $3$-graph on vertex sets $X_1, X_2$ and $X_3$ that satisfy
$m\leq\left|X_i\right|\leq 2m$ for each $i \in[3]$.
For \( i, j \in [3] \), let \( V_i \subseteq X_i \), \( V_j \subseteq X_j \) be any subsets satisfying \( |V_i|, |V_j| \geq dm \), we have the following holds. For any integer \( m' \) with \( 1 + 0.1|i - j| \leq m' \ll m \), the graph \( G \) contains a loose path of length \( m' \) whose endpoints lie in \( V_i \) and \( V_j \), respectively.
\end{lemma}
\begin{proof}
We divide the proof into the following cases according to $m'$ and the values of $i,j$.

\begin{case}
$m'$ is odd and $i\neq j$.    
\end{case}
Let $\{i,j,k\}=\{1,2,3\}.$ Note that there exists a bipartite graph $H'$ with vertex set $X_i$ and $X_j$, such that the number of edges of $H'$ equals to $|X_k|$ and $H'$ contains a path with length $m'$ and endpoints $v_1,v_2$, moreover $H'$ is $\mu$-separable for some $\mu\ll d$. Using Lemma \ref{weakblowup}, we can embed a $1$-expansion of $H'$ into $G$. Moreover, $v_1,v_2$ are embedded into $V_i$ and $V_j$. Thus, we find a loose path with length $m'$, and endpoints $v_1\in V_i$ and $v_2\in V_j$.

\begin{case}
$m'$ is even and $i\neq j$.    
\end{case}
For a vertex $v_1 \in V_i$, we consider the set of vertices in $X_k$ that, together with $v_1$ and some vertex in $V_j$, form a hyperedge. Denote this set by $V_k$. We then embed a loose path of length $m'-1$ whose endpoints $v'_1 \in V_k$ and $v_2 \in V_j$. This is possible since $G$ is $(\varepsilon, d)$-(half)-superregular, implying that $|V_k|\geq d|X_k|$. By the same reasoning as in Case 1, such a path exists. Finally, adding the hyperedge formed by $v_1$, $v'_1$, and a suitable vertex in $V_j$ completes the desired path.

\begin{case}
$m'$ is even and $i=j$.    
\end{case}
Let \( \{i, k, l\} = \{1, 2, 3\} \).  
For a vertex \( v_1 \in V_i \), we consider the vertices in \( X_k \) (with \( k \neq i \)) that, together with \( v_1 \) and some vertex in \( V_l \), form a hyperedge. Denote this set by \( V_k \).  
We then embed a loose path of length \( m'-1 \) with endpoints \( v'_1 \in V_k \) and \( v_2 \in V_j \).  
Since \( G \) is \((\varepsilon, d)\)-(half)-superregular and \( |V_k| \geq d|X_k| \), such a path exists by the same argument as in Case 1.  
Finally, adding the hyperedge formed by \( v_1 \), \( v'_1 \), and a suitable vertex in \( V_j \) completes the required path.

\begin{case}
$m'$ is odd and $i=j$.    
\end{case}
Let \( \{i, k, l\} = \{1, 2, 3\} \).  
For a vertex \( v_1 \in V_i \), we consider a 2-edge loose path of the form \( (v_1yz)(zx'y') \), where \( y, y' \in X_k \), \( x \in X_i \), and \( z \in X_l \).  
Let \( V_k \) be the set of vertices in \( X_k \) (with \( k \neq i \)) that can serve as \( y' \).  
We then embed a loose path of length \( m'-2 \) with endpoints \( v'_1 \in V_k \) and \( v_2 \in V_j \).  
Since \( G \) is \((\varepsilon, d)\)-(half)-superregular and \( |V_k| \geq d|X_k| \), such a path exists by the same reasoning as in Case 1.  
Combining this path with the initial 2-edge loose path yields the desired path.

\end{proof}

\section{Proof of Theorem \ref{oddandeven}}\label{section-proof}
In this section, we present the proof of our main result, Theorem \ref{oddandeven}.


This subsection is devoted to the proof of Theorem \ref{oddandeven}.
Following the strategy of Khan \cite{khan2011spanning}, we first simplify the problem to the case where each loose cycle in $\mathcal{C}$ has at most $\log n$ vertices, as stated in Theorem \ref{reducethm0}. 
The complementary case, where $\mathcal{C}$ contains a loose cycle with more than $\log n$ vertices, is treated separately in Appendix \ref{appendix}.
The main result of this section is the following theorem.

\begin{theorem}\label{reducethm0}
For every $\eta > 0$, there exists an integer $n_0 = n_0(\eta)$ such that the following holds for any $n \geq n_0$.  
Let $\mathcal{C}$ be a $3$-graph consisting of vertex-disjoint loose cycles $C_{n_1}, C_{n_2}, \ldots, C_{n_r}$, where $\sum_{i=1}^{r} n_i = n$ and $n_i \leq \log n$ for each $i \in [r]$. Let $k$ denote the number of loose cycles with odd lengths.  
If $\mathcal{H}$ is an $n$-vertex $3$-graph with $\delta_2(\mathcal{H}) \geq (n + 2k)/4 + \eta n$, then $\mathcal{H}$ contains $\mathcal{C}$ as a spanning subhypergraph.
\end{theorem}

In the remainder of the paper, we fix positive constants that satisfy the following hierarchy:  
\begin{equation}\label{equa1}
0 < 1/n \ll 1/L_0 \ll \varepsilon \ll d \ll \psi \ll \eta \ll 1,    
\end{equation} 
where $L_0$ is an integer and we choose these constants successively from right to left.
Let $\mathcal{C}$ be a $3$-graph consisting of vertex-disjoint loose cycles $C_{n_1}, C_{n_2}, \ldots, C_{n_r}$, where $\sum_{i=1}^{r} n_i = n$ and $ n_i \leq \log n$ for each $i \in [r]$. 
The number of loose cycles with odd lengths in $\mathcal{C}$ is denoted by $k$.
Suppose that $\mathcal{H}$ is an $n$-vertex $3$-graph with minimum codegree at least $(n + 2k)/4 + \eta n$.
Before proceeding to the main proof, we first prove three special cases. Let $c_x$ denote the number of loose cycles with order $x$. 

\begin{proposition}\label{SCS}
If one of the following three conditions holds, then Theorem \ref{reducethm0} holds.\\
$(1)$ There exists an integer $s$ such that $c_{s}\geq (1-\eta/2)n/s$;\\
$(2)$ $k\geq (1-\eta/5)n/6$;\\
$(3)$ $8c_8+12c_{12}\geq (1-\eta^2)n$.
\end{proposition}
\begin{proof}
(1) If there exists an integer $s$ such that $c_s\geq(1-\eta/2)n/s$, then we have \[|V(\mathcal{C})\setminus \cup_{n_i=s}V(C_{n_i})|<
n-s(1-\eta/2)n/s =\eta n/2.\]
We can greedily embed all other loose cycles in $\mathcal{C}$ into $\mathcal{H}$ except those of order $s$ by Lemma \ref{pancyclic}. 
Let $\mathcal{H}'$ denote the resulting graph by removing these vertices from $\mathcal{H}$. 
Then, we have $\delta_{2}(\mathcal{H}')\geq(n+2k)/4+\eta n/2$.
By Theorem \ref{MYC}, $\mathcal{H}'$ contains a $C_s$-factor and thus Theorem \ref{reducethm0} holds.  

(2) If $k\geq(1-\eta/5)n/6$, then we have $c_6 \geq (1-\eta/2)n/6$ as $6c_6+10(k-c_6)\leq n$. Thus Theorem \ref{reducethm0} holds by (1).  

(3) Otherwise, suppose $8c_8+12c_{12}\geq(1-\eta^2)n$.
Then we have $c_8\geq\eta^2 n$ and $c_{12}\geq\eta^2 n$. 
If $c_8<\eta^2 n$ or $c_{12}<\eta^2 n$, then we have $c_8\geq(1-13\eta^2)n/8$ or $c_{12}\geq(1-9\eta^2)n/12$. There is a contradiction with (1).
By Lemma \ref{pancyclic}, we can greedily embed all other loose cycles into $\mathcal{H}$ except those of order 8 and 12. 
Denote the resulting subgraph by $\mathcal{H}_1$.
We randomly partition $V(\mathcal{H}_1)$ into $V_1$ and $V_2$, with $|V_1| = 8c_8$ and $|V_2| = 12c_{12}$. By a similar argument as in Appendix \ref{appendix}, there exists a partition $V = V_1 \cup V_2$ such that  
$$\delta_2(\mathcal{H}[V_1]) \geq (1/4 + \eta/3)|V_1| \quad \text{and} \quad \delta_2(\mathcal{H}[V_2]) \geq (1/4 + \eta/3) |V_2|.$$  
Then, $C_8$ can be embedded into $\mathcal{H}[V_1]$ and $C_{12}$ can be embedded into $\mathcal{H}[V_2]$ by Theorem \ref{MYC}.
Thus Theorem \ref{reducethm0} holds.  
\end{proof}   

\subsubsection{Step 1}
We first apply the weak hypergraph regularity lemma (Lemma \ref{degweakreg}) with parameters $\varepsilon, L_0, d$ satisfying (\ref{equa1}). By Lemma \ref{degweakreg}, we obtain a partition of $V(\mathcal{H})$ into disjoint clusters $V_0, V_1, \ldots, V_L$ such that $\left|V_0\right| \leq \varepsilon n$, and $\left|V_1\right|=\cdots=\left|V_L\right|=: m$. Note that $m$ satisfies $0.9n/L_0<m<1.1n/L_0$ and $L_0\ll n/L_0\ll n$.
Additionally, we obtain a spanning subhypergraph $\mathcal{H}'$ of $\mathcal{H}$ satisfying the following properties:

\begin{enumerate}[$(i)$]
\item $d_{\mathcal{H'}}(v)>d_{\mathcal{H}}(v)-(d+\varepsilon) n^2$ for each $v \in V$;
\item Every edge of $\mathcal{H'}$ with more than one vertex in a single cluster $V_i$ for some $i \in[L]$ has at least one vertex in $V_0$;
\item For all triples $\{i, j, k\} \in\binom{L}{3}$, we have that $\left(V_i, V_j, V_k\right)_{\mathcal{H'}}$ is either empty or $(\varepsilon, d)$-regular.
\end{enumerate}

The reduced hypergraph $\mathcal{R}$ has vertices $V_1, \ldots, V_L$ and hyperedges between $V_{i}, V_{j}, V_{k}$ precisely when $\left(V_{i}, V_{j}, V_{k}\right)_{\mathcal{H}'}$ is $(\varepsilon, d)$-regular for $\{i,j,k\}\in\binom{L}{3}$. By Lemma \ref{degreeinherit}, all but at most $36d^{1/3}L^2$ of the $2$-tuples of vertices in $\mathcal{R}$ have degree at least $\left(\frac{n+2k}{4n}+\eta -4d^{1/3}\right)L$.

Let $q, p$ be positive integers with $q>p$. Define the $3$-graph $\mathcal{A}_{p, q}^3$ as follows. The vertex set $V\left(\mathcal{A}_{p, q}^3\right)$ is the union of disjoint sets $A_1, \ldots, A_{q-p}$ and $B$, where $\left|A_j\right|=2$ for each $j \in[q-p]$ and $|B|=2p$. Any $3$-tuple of the form $\{x\} \cup A_j$ with $x \in B$ is an edge of $\mathcal{A}_{p, q}^3$ for $j \in\left[q-p\right]$. In particular, we have $\left|\mathcal{A}_{p, q}^3\right|=2q$. 
Mycroft \cite{Myc2016JCTA} established a more generalized result holding for all $k$-graphs, but we present here only the specialized version for 3-graphs.

\begin{lemma}[\cite{Myc2016JCTA}]\label{H8exist}
Suppose that $0<1/m\ll\theta\ll\psi\ll 1/q, 1/p, 1/3$ and that $G$ is a $3$-graph on vertex set $[m]$ such that $\operatorname{deg}_G(S)>(p/q+\theta) m$ for all but at most $\theta m^{2}$ sets $S\in\binom{[m]}{2}$. Then $G$ admits an $\mathcal{A}_{p, q}^3$-tiling $\mathcal{F}$ such that $|V(\mathcal{F})|\geq(1-\psi) m$.
\end{lemma}

Applying Lemma \ref{H8exist} to $\mathcal{R}$, we obtain the following result.

\begin{claim}\label{H8tiling}
There exists an $\mathcal{A}_{a, 2a+b}^3$-tiling $\mathcal{M}$ in $\mathcal{R}$ such that all but at most $\psi L$ vertices in $\mathcal{R}$ are covered by $\mathcal{M}$.
\end{claim}

By moving at most $\psi Lm\leq \psi n$ vertices into $V_0$ from uncovered clusters, we obtain an $\mathcal{A}_{a, 2a+b}^3$-factor in $\mathcal{R}$, which we still denote by $\mathcal{M}$. Now, we have $|V_0|\leq(\varepsilon+\psi )n$.

For each $\mathcal{A}_{a, 2a+b}^3 \in \mathcal{M}$, we denote the clusters in $\mathcal{A}_{a, 2a+b}^3$ by $A_{i,1}$ and $A_{i,2}$ for each $i \in [a+b]$, and by $B_j$ for each $j \in [2a]$. 
Let $B$ be the union of all $B_j$ for $j \in [2a]$. 
We claim that each 3-partite 3-graph $(A_{i,1}, A_{i,2}, B)_{\mathcal{H}'}$ is $(2a \varepsilon, d/2a)$-half-regular. Since for any subsets $A_1 \subseteq A_{i,1}$, $A_2 \subseteq A_{i,2}$, and $B_0 \subseteq B$ satisfy

\[
|A_1|\geq 2a\varepsilon |A_{i,1}|,\quad|A_2|\geq 2a\varepsilon |A_{i,2}|\quad\text{and}\quad|B_0| \geq 2a \varepsilon |B|,
\]
we have
\[
d(A_1, A_2, B_0) = \frac{\sum_{j \in [2a]} e(A_1, A_2, B_0 \cap B_j)}{|A_1||A_2||B_0|}\geq \frac{d|A_1||A_2|}{2a|A_1| |A_2|} = d/2a.
\]
The inequality holds because for any $j \in [2a]$ if $|B_0 \cap B_j| \geq \varepsilon |B_j|$, then $e(A_1, A_2, B_0\cap B_j)\geq d|A_1||A_2||B_0 \cap B_j|$. When $|B_0| \geq 2a \varepsilon |B|$ there exists $j'\in [2a]$ such that $|B_0\cap B_{j'}|\geq |B_0|/2a\geq\varepsilon |B_{j'}|$, thus $\sum_{j\in [2a]} e(A_1,A_2,B_0\cap B_j)\geq e(A_1,A_2,B_0\cap B_{j'})\geq d|A_1||A_2||B_0\cap B_{j'}|$ and $|B_0\cap B_{j'}|/|B_0| \geq 1/2a$, which implies the inequality.

By moving, if necessary, at most $(a+b)$ vertices from each cluster into $V_0$, we ensure that the number of vertices in each cluster becomes divisible by $(a+b)$.
We define the size of each cluster as $m_1$. 
By Lemma \ref{slicinglemma}, we know that $(A_{i,1}, A_{i,2}, B)_{\mathcal{H}'}$ is $(4a\varepsilon, d/4a)$-half-regular.
Now, we uniformly and randomly partition $B$ into $2(a + b)$ subsets with equal size, labeling them as $Y_{i,1}$ and $Y_{i,2}$ for each $i\in[a+b]$. 
Meanwhile, for each \( i \in [a + b] \), we uniformly and randomly partition every \( A_{i,1} \) into two subsets \( X_{i,1} \) and \( Z_{i,2} \), and partition every \( A_{i,2} \) into two subsets \( X_{i,2} \) and \( Z_{i,1} \) such that \( |X_{i,1}| : |Z_{i,2}| = |X_{i,2}| : |Z_{i,1}| = a : b \).  
For each \( \mathcal{A}^3_{a,2a+b} \in \mathcal{M} \), after the random partitioning, we obtain \( 2(a + b) \) triples \( (X_{i,j}, Y_{i,j}, Z_{i,j}) \) for $i\in [a + b]$ and $j\in[2]$, where  
\[
|X_{i,j}| = |Y_{i,j}| = \frac{a m_1}{a + b} \quad \text{and} \quad |Z_{i,j}| = \frac{b m_1}{a + b}.
\]
Denote $N=2(a+b)|\mathcal{M}|$.
Take all these triples together and relabel them by 
\( (X_i, Y_i, Z_i) \) for \( i \in [N] \). 
By Lemma~\ref{slicinglemma} and the union bound argument, there exists a partition such that the 3-partite 3-graph \( (X_i, Y_i, Z_i)_{\mathcal{H}'} \) is \((8a\varepsilon, d/8a)\)-half-regular for each \( i \in [N] \), and we fix such a partition. Thus, after this random partitioning, we obtain $N$ disjoint \((8a\varepsilon, d/8a)\)-half-regular triples \( (X_i, Y_i, Z_i) \) that cover almost all vertices of \( \mathcal{H} \). In what follows, we denote these triples by $T_i:=(X_i,Y_i,Z_i)$ for $i\in[N]$.

We now establish an essential property regarding common neighbors that will be helpful when incorporating the exceptional set $V_0$, analogous to Proposition 7.1 in \cite{kuhn2006loose}.

\begin{claim}\label{commonneighbor}
Let $\eta_1 = \eta - 4d^{1/3}$ and $\eta'_1 = 0.99\eta_1$. For every pair of vertices $u, v \in V(\mathcal{H})$, there are at least $\frac{\eta'_1 N}{2(2a+b)}$ indices $i$ for which 
\[
|N_{\mathcal{H}}(u, v) \cap Z_i| \geq \frac{\eta'_1 |Z_i|}{2(2a+b)}.
\]
\end{claim}
\begin{proof}
Recall that $m$ is the size of cluster $V_i$ obtained by the regularity lemma.
Let $u,v$ be any two vertices in $\mathcal{H}$. We define a cluster $V_i$ ($i\in[L]$) to be \textit{$(u,v)$-good} if
\[|N_{\mathcal{H}}(u, v)\cap V_i|\geq \frac{\eta_1 m}{2((2a+b)-a)}.\]
Let $N'_{uv}$ denote the number of elements in the $\mathcal{A}^3_{a,(2a+b)}$-tiling $\mathcal{M}$ which contain at least $2a + 1$ $(u,v)$-good clusters. Then we have
\[
\frac{n + 2k}{4} + \eta_1 n \leq |N_{\mathcal{H}}(u, v)| \leq 2(2a+b)m N'_{uv} + |\mathcal{M}| \left(2am + 2((2a+b)-a) \frac{\eta_1 m}{2(2a + b) - 2a}\right).
\]
Since $n \geq 2(2a+b) |\mathcal{M}| m$, it follows that
\[
N'_{uv} \geq \frac{(2(2a+b) - 1) \eta_1 |\mathcal{M}|}{2(2a+b)}.
\]

Consider any $\mathcal{A}^3_{a,(2a+b)} \in \mathcal{M}$ which contains at least $2a + 1$ (original) clusters $V_t$ such that  
\[
|V_t\cap N_{\mathcal{H}}(u, v)|\geq\frac{\eta_1 m}{2(2a + b) - 2a}.
\]
Thus, for some $A_i$ in this $\mathcal{A}^3_{a,(2a+b)}$, we have
\[
|A_i\cap N_{\mathcal{H}}(u, v)|\geq\frac{\eta_1 m}{2(2a + b) - 2a}-(a+b).
\]
Since $A_i$ contains a subset of the form $Z_i$, applying Lemma \ref{fairly}, we conclude that with high probability, the intersection of $N_{\mathcal{H}}(u, v)$ with $Z_i$ has size at least
\[
\frac{\eta_1 |Z_i|}{2(2a+b)}-(a+b)>\frac{\eta'_1 |Z_i|}{2(2a+b)}.
\]
Thus, there exist at least $N' _{uv}\geq \frac{\eta'_1 N}{2(2a+b)}$ sets $Z_i$, each containing at least $\frac{\eta'_1 |Z_i|}{2(2a+b)}$ neighbors of the pair $u, v$, as required.
\end{proof}

By Lemma \ref{regtosupreg}, after removing at most $8a\varepsilon mL\leq 8a\varepsilon n$ vertices into $V_0$, the $3$-partite $3$-graph $(X_i, Y_i, Z_i)_{\mathcal{H'}}$ becomes $(16a\varepsilon, d/16a)$-half-superregular for each $i\in[N]$. Note that in this progress, $|X_i|:|Y_i|:|Z_i|$ doesn't changes when applying Lemma \ref{regtosupreg}.

Next, for each $i\in [N]$, we remove $\sqrt{m}$ additional vertices from $Z_i$ to $V_0$ (retaining the notation $(X_i, Y_i, Z_i)$ for the resulting triple). By Lemma \ref{slicinglemma} $(ii)$, we conclude that for each $i\in[N]$, the $3$-partite $3$-graph $(X_i, Y_i, Z_i)_{\mathcal{H'}}$ is $(32a\varepsilon, d/32a)$-half-superregular. 
Let $\varepsilon_1:=32a\varepsilon$ and $d_1:=d/32a$.

For each $i\in[N]$, we have
\[
|Z_i|+\sqrt{m} =\frac{b}{a}|X_i|=\frac{b}{a}|Y_i|.
\]
By Proposition \ref{SCS}, We may assume that $k< (1-\eta/5)n/6$. Note that $\frac{2n-4k-0.04\eta n}{n+2k+0.02\eta n}$ is a decrease function of $k$ and $\eta\ll 1$. Then
\[
\frac{b}{a}\geq \frac{2n-4k-0.04\eta n}{n+2k+0.02\eta n}>1+\eta^2.
\] 
Rearranging the terms, we have 
\[
|Z_i|-|X_i|\geq (1-\varepsilon)m_1\left(\frac{b-a}{a+b}\right)-\sqrt{m} \geq \eta^3 m_.
\]

Let $\eta' = \eta'_1 - 2\varepsilon$. Then, for every pair of vertices $u, v \in V(\mathcal{H})$, there are at least $\frac{\eta' N}{2(2a+b)}$ indices $i$ such that 
\[
|N_{\mathcal{H}}(u, v) \cap Z_i| \geq  \frac{\eta'|Z_i|}{2(2a+b)}.\]

Now, the number of vertices in $V_0$ is at most
\[
|V_0|\leq(8a \varepsilon+\psi)n+(a+b)L+N\sqrt{m}.
\]
Defining 
\[
\xi=\frac{(8a \varepsilon+\psi)n+(a+b)L+N\sqrt{m}}{n},
\]
we obtain $|V_0|\leq\xi n$ and $\xi\ll\eta', \eta$. Thus, almost all vertices, except for those in $V_0$, are covered by triples $T_1, T_2, \ldots, T_N$.

\subsubsection{Step 2}
We now proceed with the detailed embedding procedure. First, we embed some loose cycles to exhaust all the vertices in $V_0$. Let $x$ be the smallest integer satisfying
$\sum_{i=1}^{x} n_i/2 \geq |V_0|$.
Next, we remove $\sum_{i=1}^{x}n_i/2- |V_0|$ vertices from some sets $Z_i$ into $V_0$ such that
\[
|V_0| = \sum_{i=1}^{x} \frac{n_i}{2}.
\]

Let $v_1, v_2, \ldots, v_t$ be an enumeration of the vertices in $V_0$. We aim to embed the loose cycles $C_{n_1}, C_{n_2}, \ldots, C_{n_x}$ to cover all the vertices in $V_0$.
We embed the loose cycle $C_{n_1}$ such that it covers the vertices $v_1, v_2, \ldots, v_{n_1/2}$. 
We claim that for every pair $v_i v_{i+1}$ with $1\leq i\leq n_1/2$ (including the pair $v_1 v_{n_1/2}$), there exist distinct vertices $w_i$ such that $v_iw_iv_{i+1}\in E(\mathcal{H})$ (and $v_1 w_{n_1/2} v_{n_1/2}\in E(\mathcal{H})$) and $w_i\in \bigcup_{j=1}^N Z_j$, where $w_i$ are pairwise distinct.
Moreover, this holds for every loose cycle $C_{n_i}$ for $i \in [x]$. 
Note that $|V(T_i)| \leq 3m$ for each $i \in [N]$. Thus, the total number of required vertices $w_i$ is at most
$|V_0| \leq \xi n \leq 4 \xi N m.$
Furthermore, we have
\[
|V_0| \leq 4 \xi N m \leq \left(\sqrt{\xi} m \right) \left(\frac{\eta' N}{2(2a+b)}\right),
\]
where the final inequality holds because $\xi \ll \eta'$ and $a, b \leq 200/\eta +1 \leq 250/\eta'$. 
Hence, by Claim \ref{commonneighbor}, we can greedily select the vertices $w_i$ such that no single $Z_i$ has more than $\sqrt{\xi} m$ vertices consumed.

The following result asserts that after removing these vertices $w_i$, the $3$-partite $3$-graph $(X_i, Y_i, Z_i)_{\mathcal{H'}}$ remains sufficiently half-superregular for each $i \in [N]$.

\begin{claim}\label{extravtx}
Let $\varepsilon^* := 1.01\varepsilon_1$ and $d^* := d_1^2 / 16$. The resulting $3$-partite $3$-graph $(X_i, Y_i, Z_i)_{\mathcal{H}'}$ is $(\varepsilon^*, d^*)$-half-superregular for every $i \in [N]$.
\end{claim}

\begin{proof}
For each $Z_i$, we randomly select a subset $Z'_i \subseteq Z_i$ with $|Z'_i| = \eta^4 |Z_i|$. 
By Lemma \ref{slicinglemma} $(iii)$, with high probability, the $3$-partite $3$-graph $(X_i, Y_i, Z_i \setminus Z'_i)_\mathcal{H'}$ is $(1.01 \varepsilon_1, d_1^2 / 16)$-half-superregular for each $i\in [N]$.
By Lemma \ref{fairly} and Claim \ref{commonneighbor}, we know that for every pair of vertices $v, w \in V(H)$, there exist at least $\frac{\eta' N}{2(2a+b)}$ indices $i$ such that
\[
|N_H(v, w) \cap Z'_i| \geq \frac{\eta' |Z'_i|}{3(2a+b)} \gg \sqrt{\xi}m.
\]
This ensures that we can greedily select at most $\sqrt{\xi}m$ vertices from these $Z'_i$ to form loose cycles together with the vertices in $V_0$, then remove the selected vertices from $Z_i'$. 
The remaining unused vertices of $Z_i'$ are returned into $Z_i$, and we continue to denote the triple as $(X_i, Y_i, Z_i)$.

For any subsets $X_i''\subseteq X_i, Y_i''\subseteq Y_i, Z_i''\subseteq Z_i$ satisfying 
\[
|X_i''| \geq 1.01\varepsilon_1|X_i|,  \quad |Y_i''| \geq 1.01\varepsilon_1 |Y_i|,  \quad |Z_i''| \geq 1.01\varepsilon_1 |Z_i| \geq \varepsilon_1 \frac{bm_1}{a+b},
\]
we have $d(X_i'', Y_i'', Z_i'')\geq  d_1\geq d_1^2 / 16$, thereby ensuring that $(X_i, Y_i, Z_i)_\mathcal{H'}$ remains $(1.01 \varepsilon_1, d_1^2 / 16)$-half-regular for every $i\in [N]$. Moreover, for any vertex in $Z_i$, its degree is at least $d_1$, and for any vertex from $X_i$ or $Y_i$, its degree is at least $d_1^2 / 16$ in $(X_i, Y_i, Z_i)_\mathcal{H'}$ by Lemma \ref{slicinglemma}. Thus, the claim holds.
\end{proof}

We still denote the triples as $T_i=(X_i, Y_i, Z_i)$. Note that $m/(a+b)<|T_i|<4m$ for each $i\in [N]$. Thus for every $i\in [N]$, we have $1/|T_i| \ll \varepsilon^*$.

For every $X_i$ and $Y_j$, we aim to find additional vertices to form hyperedges together with one vertex in $X_i$ and another in $Y_j$. The following claim demonstrates that there exists sufficient connectivity between any two triples. While $\sqrt{n}$ is sufficient for our purposes, this bound could be improved to a linear term in $n$.
Given a 3-graph $\mathcal{H}$ and a vertex set $Z\subseteq V(\mathcal{H})$, we write $\deg_{\mathcal{H}}(x, y; Z)$ for the number of vertices $z\in Z$ such that $(xyz)\in E(\mathcal{H})$.

\begin{claim}\label{connect}

For any $X_i$ and $Y_j$ with $i,j\in[N]$, there exists a triple $T_h$ and $X^0_i \subseteq X_i$ with $|X^0_i| \geq 0.001 |X_i|$ such that the following holds. For any $x\in X_i^0$, there exists a subset $Y^0_j \subseteq Y_j$ with $|Y^0_j| \geq 0.001 |Y_j|$ such that for any $y\in Y_j^0$, we have $\deg_{\mathcal{H}}(x, y; V(T_h))\geq \sqrt{n}$.
\end{claim}

\begin{proof}
Let $\widetilde{T}=\bigcup_{i=1}^{N} T_i$ and $X_iY_j=\{xy: x \in X_i, y \in Y_j \}$. Consider a bipartite graph $B$ with vertex set $V\left(X_iY_j, V(\widetilde{T})\right)$. An edge exists between a vertex pair $x y$ and a vertex $t \in \widetilde{T}$ in $B$ if and only if $(x y t)\in E(\mathcal{H})$. According to the codegree condition of $\mathcal{H}$, we know that in the bipartite graph $B$, every vertex in the set $X_iY_j$ has degree at least 
$\left(1/4 + \eta -2 \xi \right) n \geq n/4 \geq |V(\widetilde{T})|/4$.
Now, we claim that the number of vertices in $\widetilde{T}$ with degree larger than $\left|X_iY_j\right|/10$ is at least $n/7$.
To prove this, assume the contrary. By double counting, we have
\[
\frac{n}{4} \left|X_i Y_j\right| \leq |E(B)| \leq \frac{n}{10} \left|X_i Y_j\right| + \frac{n}{7} \left|X_i Y_j\right|< \frac{n}{4} \left|X_i Y_j\right|.
\]
This leads to a contradiction. Thus, there exists $h$ such that the number of vertices in $T_h$ with degree larger than $\left|X_iY_j\right|/10$ is at least $n/(8N)$. 
We claim that there are at least $|X_iY_j|/100$ elements in $X_iY_j$ such that their codegree in $T_h$ is larger than $\sqrt{n}$. Otherwise, a contradiction holds:
\[
\begin{aligned}
\frac{n}{8N} \times \frac{\left|X_iY_j\right| }{10} \leq |E(X_iY_j,T_h)| \leq \sqrt{n} \times \left|X_iY_j\right| + \frac{n}{N} \times \frac{\left|X_iY_j\right| }{100} \left|X_iY_j\right| < \frac{n}{90N} \left|X_iY_j\right|.
\end{aligned}
\]

Now consider a bipartite graph $B'$ with vertex set $V(X_i, Y_j)$, and its edge density is no less than $1/100$. We claim that there are $0.001 |X_i|$ vertices in $X_i$ such that their degree in $B'$ is no less than $0.001 |Y_j|$. Otherwise, the edge density of $B'$ is no more than $1 \times 0.001 + 0.001 \times 1 < 1/100$, which leads to a contradiction. Let these vertices be $X^0_i$ then the claim holds.
\end{proof}

To ensure consistency in the number of points in subsequent work, we should rearrange the order of every triple to be even.

\begin{claim}\label{turneven}
We can embed at most $N/2$ cycles of length greater than $3$ using only the vertices from the triples such that after removing the vertices, the number of remaining vertices in each triple is even.
\end{claim}

\begin{proof}
Since $N$ and $\sum_{i\in N}|T_i|$ are even, the number of triples with an odd order is also even. We use the following progress until there is no odd order triple.
If there exist $T_i$ and $T_j$ having odd orders, by Claim \ref{connect}, there exists a vertex $x$ in $T_i$, a subset of vertices $Y'_j \subseteq Y_j$ with order no less than $0.001|Y_j|$ as well as a triple $T_p$, such that for any $y \in Y'_j$, there exist $\sqrt{n}$ vertices $t_j \in T_p$ such that $(x t_j y_j) \in E(\mathcal{H})$. 
 
Applying Lemma \ref{pathembed}, we can construct a loose path with both of its endpoints $y_1, y_2$ in $Y_j$ of a certain length (more than $2$) and there exists $t_1,t_2 \in T_p$ (since $\sqrt{n} \gg (N /2)  \log n$, there are always sufficiently many such vertices available) such that $(x t_1 y_1),(x t_2 y_2) \in E(\mathcal{H})$. This path, along with the vertices $x$ and 2 vertices $t_1, t_2$ in $T_p$, form a loose cycle $C_{ij}$ of length more than $3$ in $\mathcal{C}$. We then remove this cycle from $\mathcal{C}$ and remove the vertices of the triples that were used in this process.

During this process, $|T_i|$ decreased by $1$, $|T_j|$ decreased by $|C_{ij}|-3$ and $|T_p|$ decreased by $2$, and the orders of the other triples remain unchanged. Thus, the parity of all triples remains unchanged while $|T_i|$ and $|T_j|$ change from odd to even. After at most $N/2$ such steps, all $|T_i|$ become even.
\end{proof}

Let $T_i = (X_i, Y_i, Z_i)$ still denote the triple for every $i \in [N]$. 
Let $C_{n_1}, C_{n_2}, \dots, C_{n_{r'}}$ be the remaining loose cycles and $k$ be the number of loose cycles with odd length. Without loss of generality, assume $n_1 \leq n_2 \leq \dots \leq n_{r'}$.
In what follows, we abbreviate $C_{n_i}$ by $C_i$. 
Let $r_1':=r'-2(N-1)$.
Let $\mathcal{C}_0:=\cup_{i=1}^{r'}C_{i}$, $\mathcal{C}_1:=\cup_{i=1}^{r_1'}C_{i}$ and $\mathcal{C}_2:=\cup_{i=r_1'+1}^{r'}C_{i}$. 
Let $C_{lgc}$ denote the longest cycle in $\mathcal{C}_1$. Denote the number of cycles in $\mathcal{C}_1$ by $r'_1$, and let $k'_1$ be the number of odd cycles in $\mathcal{C}_1$.

\begin{claim}[Pre-partition]\label{yuhuafen}
    There exists a partition $\mathcal{C}_1 = \cup_{i=1}^{N} \overline{\mathcal{F}}_i$ with $\overline{\mathcal{F}}_i \cap \overline{\mathcal{F}}_j = \emptyset$ when $i \neq j$ such that for any $i \in [N]$, $\big|\sum_{j \in \overline{\mathcal{F}}_i} |C_j|-|T_i|\big|\leq N^2\log n$. Moreover, the number of odd cycles in any $\overline{\mathcal{F}}_i$ is at most $k(1+\eta^{3.2})/N+\sqrt{n}$.
\end{claim}
\begin{proof}
    Without loss of generality, we prove a stronger version of the result. Note that for any two triples, the difference in their orders is no more than $\eta^4m + N\log n + 4n \leq \eta^{3.5} m$, and for each pair of cycles, the difference in their orders is bounded by $|C_{lgc}| \leq \log n$. Given a partition $\mathcal{P}: \mathcal{C}_1 = \cup_{i=1}^{N} \overline{\mathcal{F}}_i$, define $\phi_i$ as the number of odd cycles in $\overline{\mathcal{F}}_i$, and define $\beta_i$ as the number of even cycles in $\overline{\mathcal{F}}_i$. By Lemma \ref{bili} (with $k'_1 = q$, $r'_1 - k'_1 = q$, and $a_i = |T_i| / \sum |T_i|$), there exists at least one partition such that $\phi_i \in \left[k'_1|T_i| / \sum |T_i| - 1, k'_1|T_i| / \sum |T_i| + 1\right]$ and $\beta_i \in \left[(r'_1 - k'_1)|T_i| / \sum |T_i| - 1, (r'_1 - k'_1)|T_i| / \sum |T_i| + 1\right]$ for every $i \in [N]$.

We claim that such a partition exists, satisfying $\phi_i \in \left[k'_1|T_i| / \sum |T_i| - 1, k'_1|T_i| / \sum |T_i| + 1\right]$ and $\beta_i \in \left[(r'_1 - k'_1)|T_i| / \sum |T_i| - 1, (r'_1 - k'_1)|T_i| / \sum |T_i| + 1\right]$ for every $i \in [N]$. Furthermore, for any $i \in [N]$, we have $\left|\sum_{j \in \overline{\mathcal{F}}_i} |C_j| - |T_i|\right| \leq N^2 \log n$. Otherwise, a contradiction arises under the following two cases. This implies that the number of odd cycles in any $\overline{\mathcal{F}}_i$ is at most $k'_1|T_i|/\sum |T_i| + 1 \leq  \frac{k(1 + \eta^{3.2})}{N}+\sqrt{n}$.

\emph{Case 1:}
For any such partition $\mathcal{P}$ we say $i$ is \emph{1-bad} if $\sum_{C \in \overline{\mathcal{F}}_i} |C|-|T_i| > N^2\log n$, and let $\mathcal{S}=\sum_{i \ \text{is} \ 1\text{-bad}}\big( \sum_{C \in \overline{\mathcal{F}}_i} |C|-|T_i| \big)$.
If for any such partition with $\phi_i \in [k'_1|T_i| / \sum |T_i| - 1, k'_1|T_i| / \sum |T_i| + 1]$ and $\beta_i \in [(r'_1-k'_1)|T_i| / \sum |T_i| - 1, (r'_1-k'_1)|T_i| / \sum |T_i| + 1]$ there exists $i \in [N]$ such that $i$ is \emph{1-bad}, choose one such partition with minimum $\mathcal{S}$. Note that there must be $i' \in [N]$ with $\sum_{C \in \overline{\mathcal{F}}_{i'}} |C|-|T_{i'}| < - 2\log n$ otherwise $\sum_{C \in \mathcal{C}_1} |C|-\sum|T_i| > (N^2-2N+2)\log n\geq \sum_{C \in \mathcal{C}_0} |C|-\sum|T_i|$, which leads to a contradiction. Thus $\big(\sum_{C \in \overline{\mathcal{F}}_{i'}} |C|+2|C_{lgc}|\big)/|T_{i'}|\leq 1.$

We claim that there exists $C_{x_i} \in \overline{\mathcal{F}}_i$ and $C_{x_{i'}} \in \overline{\mathcal{F}}_{i'}$ such that $|C_{x_i}| > |C_{x_{i'}}|$ and $C_{x_i},C_{x_{i'}}$ have same parity. Let $C^i_{lgoc}$ be the longest odd cycle in $ \overline{\mathcal{F}}_i$ and $C^i_{lgec}$ be the longest even cycle in $ \overline{\mathcal{F}}_i$. Otherwise since $(\phi_i-1)/(\phi_{i'}+1)\leq |T_i|/|T_{i'}|$ and $(\beta_i-1)/(\beta_{i'}+1)\leq |T_i|/|T_{i'}|$, we have 
\[
\frac{\sum_{C \in \overline{\mathcal{F}}_i} |C| - 2|C_{lgc}|}{\sum_{C \in \overline{\mathcal{F}}_{i'}} |C| + 2|C_{lgc}|} 
\leq \frac{(\phi_i - 1) |C^i_{lgoc}| + (\beta_i - 1) |C^i_{lgec}|}{(\phi_{i'} + 1) |C^i_{lgoc}| + (\beta_{i'} + 1) |C^i_{lgec}|}
\leq \frac{|T_i|}{|T_{i'}|}.
\] 
Hence, \[
\frac{\big(\sum_{C \in \overline{\mathcal{F}}_i} |C|-2|C_{lgc}|\big)}{|T_i|}
\leq \frac{\big(\sum_{C \in \overline{\mathcal{F}}_{i'}} |C|+2|C_{lgc}|\big)}{|T_{i'}|}\leq 1,
\]
which implies $N^2\log n<\sum_{C \in \overline{\mathcal{F}}_i} |C|-|T_i| \leq 2|C_{lgc}|$, and this leads to a contradiction. Thus, delete $C_{x_i} $ from $\overline{\mathcal{F}}_i$ and add it into $\overline{\mathcal{F}}_{i'}$ while delete $C_{x_{i'}} $ from $\overline{\mathcal{F}}_{i'}$ and add it into $\overline{\mathcal{F}}_{i}$ makes $\mathcal{S}$ decrease by at least $2$, and this leads to a contradiction with the minimality of $\mathcal{S}$.

\emph{Case 2:}
For any such partition $\mathcal{P}$ we say $i$ is \emph{2-bad} if $\sum_{C \in \overline{\mathcal{F}}_i} |C|-|T_i| <- N^2\log n$, and let $\mathcal{S'}=\sum_{i \ \text{is} \ 2\text{-bad}}\big( \sum_{C \in \overline{\mathcal{F}}_i} |C|-|T_i| \big)$.
If for any such partition with $\phi_i \in [k'_1|T_i| / \sum |T_i| - 1, k'_1|T_i| / \sum |T_i| + 1]$ and $\beta_i \in [(r'_1-k'_1)|T_i| / \sum |T_i| - 1, (r'_1-k'_1)|T_i| / \sum |T_i| + 1]$ there exists $i \in [N]$ such that $i$ is \emph{2-bad}, choose one such partition with maximum $\mathcal{S'}$. Note that there must be $i' \in [N]$ with $\sum_{C \in \overline{\mathcal{F}}_{i'}} |C|-|T_{i'}| > 2\log n$ otherwise $\sum_{C \in \mathcal{C}_1} |C|-\sum|T_i| < -(N^2-2N+2)\log n\leq -2N\log n<-\sum_{C \in \mathcal{C}_2} |C|$, thus $\sum_{C \in \mathcal{C}_0} |C|-\sum|T_i|<0$, which leads to a contradiction. Thus $\big(\sum_{C \in \overline{\mathcal{F}}_{i'}} |C|-2|C_{lgc}|\big)/|T_{i'}|\geq 1.$

We claim that there exists $C_{x_i} \in \overline{\mathcal{F}}_i$ and $C_{x_{i'}} \in \overline{\mathcal{F}}_{i'}$ such that $|C_{x_i}| < |C_{x_{i'}}|$ and $C_{x_i},C_{x_{i'}}$ have the same parity. Let $C^i_{stoc}$ be the shortest odd cycle in $ \overline{\mathcal{F}}_{i}$ and $C^i_{stec}$ be the shortest even cycle in $ \overline{\mathcal{F}}_i$. Otherwise, since $(\phi_i+1)/(\phi_{i'}-1)\geq |T_i|/|T_{i'}|$ and $(\beta_i+1)/(\beta_{i'}-1)\geq |T_i|/|T_{i'}|$, we have
\[
\frac{\big(\sum_{C \in \overline{\mathcal{F}}_i} |C|+2|C_{lgc}|\big)}{\big(\sum_{C \in \overline{\mathcal{F}}_{i'}} |C|-2|C_{lgc}|\big)}
\geq\frac{\big((\phi_i+1)\times |C^i_{stoc}|+(\beta_i+1)\times |C^i_{stec}|\big)}{\big((\phi_{i'}-1) \times|C^i_{stoc}|+(\beta_{i'}-1) \times|C^i_{stec}|\big)}
\geq \frac{|T_i|}{|T_{i'}|}.
\]

Hence, 
\[
\frac{\big(\sum_{C \in \overline{\mathcal{F}}_i} |C|+2|C_{lgc}|\big)}{|T_i|}
\geq \frac{\big(\sum_{C \in \overline{\mathcal{F}}_{i'}} |C|-2|C_{lgc}|\big)}{|T_{i'}|}\geq 1,
\]
which implies $-2|C_{lgc}|\leq\sum_{C \in \overline{\mathcal{F}}_i} |C|-|T_i|<-N^2\log n$, this leads to a contradiction. Thus delete $C_{x_i} $ from $\overline{\mathcal{F}}_i$ and add it into $\overline{\mathcal{F}}_{i'}$ while delete $C_{x_{i'}} $ from $\overline{\mathcal{F}}_{i'}$ and add it into $\overline{\mathcal{F}}_{i}$ makes $\mathcal{S'}$ increase by at least $2$, and moreover since $\sum_{C \in \overline{\mathcal{F}}_i} |C|-|T_i| <- N^2\log n+\log n<0$, by this step $\mathcal{S}$ doesn't increase. This leads to a contradiction with the maximality of $\mathcal{S'}$. 
\end{proof}
\begin{claim}\label{embedlast}
$(1)$ If $|C_{lgc}|\geq 16$, then there exists an embedding $\Psi$ from $\mathcal{C}_2$ into $\cup^N_{i=1}T_i$ such that the following holds. 
We can find a partition $\mathcal{C}_1= \cup_{i=1}^{N} \mathcal{F}_i$ with $\mathcal{F}_i \cap \mathcal{F}_j = \emptyset$ when $i \neq j$ such that for any $i \in [N]$, $|V(T_i\setminus\cup_{C\in\mathcal{C}_2}\Psi(V(C))|=\sum_{C \in \mathcal{F}_i} |C|$. Moreover, the number of odd cycles in any $\mathcal{F}_i$ is no more than $\frac{k (1 + \eta^3)}{N}+2\sqrt{n}$.

$(2)$ If $|C_{lgc}|< 16$, there exists a partition $\mathcal{C}_0 = \cup_{i=1}^{N} \mathcal{F}_i$ with $\mathcal{F}_i \cap \mathcal{F}_j = \emptyset$ when $i \neq j$ such that for any $i \in [N]$, $|V(T_i)|=\sum_{C \in \mathcal{F}_i} |C|$. Moreover, the number of odd cycles in any $\mathcal{F}_i$ is no more than $\frac{k (1 + \eta^3)}{N}$.
\end{claim}
\begin{proof}
(1) By Claim \ref{yuhuafen}, there exists a partition $\mathcal{C}_1 = \cup_{i=1}^{N} \overline{\mathcal{F}}_i$ with $\overline{\mathcal{F}}_i \cap \overline{\mathcal{F}}_j = \emptyset$ when $i \neq j$ such that for any $i \in [N]$, $\big|\sum_{j \in \overline{\mathcal{F}}_i} |C_j|-|T_i|\big|\leq N^2\log n$. Moreover, the number of odd cycles in any $\overline{\mathcal{F}}_i$ is no more than $\frac{k (1 + \eta^{3.2})}{N}+\sqrt{n}$.

For any $i \in [N-1]$, by Claim \ref{connect}, there exists a triple $T_l$ (and $T_{l'}$) with $l = l(i), l' = l'(i)$, and subsets $X_i^0 \subseteq X_i$ (and $Y_i^0 \subseteq Y_i$) where $|X_i^0| \geq 0.001 |X_i|$ (and $|Y_i^0| \geq 0.001 |Y_i|$), such that for any $x_i \in X_i^0$ (and $y_i \in Y_i^0$), there exist at least $0.001 |Y_{i+1}|$ (and $0.001 |X_{i+1}|$) vertices $y_{i+1} \in Y_{i+1}$ (and $x_{i+1} \in X_{i+1}$) such that for each $y_{i+1}$ (and $x_{i+1}$), there are at least $\sqrt{n}$ values of $t$ in $T_l$ (and $T_{l'}$) with $(x_i t y_{i+1}) \in E(\mathcal{H})$ (and $(x_{i+1} t y_i) \in E(\mathcal{H})$). Since $\sqrt{n} \gg 2N \log n$, there are always sufficiently many such vertices $t$ available. For each triple $T_i$, let $w(i)=\sum_{j\in [N-1]}\big(\chi(l(j)=i)+\chi(l'(j)=i)\big)$, where $\chi(\cdot)$ is the indicator function, which takes a value of 1 when its argument is true and 0 otherwise. Let $t(i)=|T_i|-2w(i)$.

We use the following cyclic process for each $i \in [N-1]$:

In step $i$, remove some cycles from $\overline{\mathcal{F}}_i$ to $\overline{\mathcal{F}}_{i+1}$ if $\sum_{C \in \overline{\mathcal{F}}_i} |C|-t(i)\geq -6$ or remove some cycles from $\overline{\mathcal{F}}_{i+1}$ to $\overline{\mathcal{F}}_i$ if $\sum_{C \in \overline{\mathcal{F}}_i} |C|-t(i)\leq -|C_{lgc}|-6$ such that the sum of their orders is no less than $t(i)-|C_{lgc}|-6$ and no more than $t(i)-6$. Denote these cycles as $\mathcal{F}_i$.

Since $\big|\sum_{C \in \overline{\mathcal{F}}_i} |C|-t(i)\big|\leq \max\{N^2\log n, i N^3 \log n\}=i N^3 \log n$, by adding/deleting no more than $\frac{i}{|C_6|}\times N^3 \log n+|C_{lgc}|+6\leq\frac{i}{5}\times N^3 \log n$ cycles into/from $\overline{\mathcal{F}}_i$ this progress could be reached. Thus the number of odd cycles in $\mathcal{F}_i$ is no more than $$\frac{k (1 + \eta^{3.2})}{N}+\sqrt{n}+(i-1)\times N^3 \log n + \frac{i}{5}\times N^3 \log n \leq \frac{k (1 + \eta^{3})}{N}+2\sqrt{n}.$$

We embed 2 paths $L'$ and $L''$ ($|L'|,|L''|\geq 3$) in $T_i$ with its end vertices $x',x''$ in $X_i^0 \cap T_i$ and $y',y''$ in $Y_i^0 \cap T_i$ such that $|L'|,|L''|=\big|t(i)-\sum_{C \in \mathcal{F}_i} |C|\big|/2$ by Lemma \ref{pathembed}. For both of $L'$'s end vertices $x'$ and $y'$, there exists $Y^0_{i+1}\subseteq Y_{i+1}$ with $|Y^0_{i+1}|\geq 0.001|Y_{i+1}|$ and $X^0_{i+1}\subseteq X_{i+1}$ with $|X^0_{i+1}|\geq 0.001|X_{i+1}|$ such that for every vertex $y^0_{i+1}$ in $Y^0_{i+1}$(and $x^0_{i+1}$ in $X^0_{i+1}$), there exist $l=l(i),l'=l'(i)$ and at least $\sqrt{n}$ numbers of $t$ in $T_l$ (and $T_{l'}$) with $(x' t y^0_{i+1})\in E(\mathcal{H})$ (and $(x^0_{i+1} t y')\in E(\mathcal{H})$). By Lemma \ref{pathembed}, embed a path $L_0'$ (where $|L_0'| \geq 3$) in $T_{i+1}$, with its end vertices $y_0'$ and $x_0'$ in $Y^0_{i+1}$ and $X^0_{i+1}$, respectively, such that $|L'| + |L_0'| + 2$ equals $|C'|$ for some $C' \in \mathcal{C}_2$. By Claim \ref{connect}, there exists $t'_1\in T_l$ and $t'_2\in T_{l'}$ such that $(x'y_0't'_1),(y'x_0't'_2)\in E(\mathcal{H})$. Thus, $L',L'_0$ together with 2 vertices $t'_1,t'_2$ form a cycle $C'$ in $\mathcal{C}_2$, then delete $C'$ in $\mathcal{C}_2$. Denote $\Psi(V(C'))$ to be the vertices in $L',L'_0$ and $t'_1,t'_2$, then remove every vertices in $\Psi(V(C'))$ from $V(\cup^N_{i=1}T_i)$. Then do a same progress for $L''$ such that $L''$ together with a path $L''_0$ ($|L_0''|\geq 3$) embedded in $T_{i+1}$ and 2 vertices in $T_l$ and $T_{l'}$ form a cycle $C''$ in $\mathcal{C}_2$, then delete $C''$ in $\mathcal{C}_2$. Denote $\Psi(V(C''))$ to be the vertices in $L'',L''_0$ and $t''_1,t''_2$, then remove every vertices in $\Psi(V(C''))$ from $V(\cup^N_{i=1}T_i)$.

These cycles exist since for any cycle $C_x$ in $\mathcal{C}_2$, $|C_x|\geq 16$, thus $2|C_x|-10\geq |C_{lgc}|+6$, and this implies that for any even number $w$ in $[t(i)-|C_{lgc}|-6, t(i)-6]$, $w/2\geq 3$ and $|C_x|-w/2\geq|C_x|-(|C_{lgc}|+6)/2\geq 5$, hence $|L'|,|L''|,|L_0'|,|L''_0|\geq 3$ holds.

Having done this step, we have 
\[
\big|\sum_{C \in \overline{\mathcal{F}}_{i+1}} |C|-t(i+1)\big|\leq i\times N^3\log n+2\log n+|C_{lgc}|+6+N^2 \log n\leq (i+1)\times N^3 \log n,
\]
and the number of odd cycles in $\overline{\mathcal{F}}_{i+1}$ is no more than $\frac{k (1 + \eta^{3.2})}{N}+\sqrt{n} +\frac{i}{5}\times N^3 \log n\leq\frac{k (1 + \eta^{3.2})}{N}+\sqrt{n} +i\times N^3 \log n.$ This implies the following Fact:

\begin{fact}\label{processfact}
    Having finished step $i$ we have $\big|\sum_{C \in \overline{\mathcal{F}}_{i+1}} |C|-t(i+1)\big|\leq (i+1) N^3 \log n$ and the number of odd cycles in $\overline{\mathcal{F}}_{i+1}$ is no more than $\frac{k (1 + \eta^{3.2})}{N}+\sqrt{n}+i N^3 \log n$.
\end{fact}

When it comes to $\mathcal{F}_N$, since every other triple is used up, its order is the union of the cycles left in $\mathcal{C}_1$. And the number of odd cycles in it is no more than $\frac{k (1 + \eta^{3.2})}{N}+\sqrt{n} + N^4 \log n \leq \frac{k (1 + \eta^{3})}{N}+2\sqrt{n}$.

(2) Recall that $c_x$ denotes the number of cycles with order $x$ and we let $c^i_x$ denote the number of cycles with order $x$ in $\overline{\mathcal{F}}_i$. Partition $\mathcal{C}_0= \cup_{i=1}^{N} \overline{\mathcal{F}}_i$ with $\overline{\mathcal{F}}_i \cap \overline{\mathcal{F}}_j = \emptyset$ when $i \neq j$ such that for any $i \in [N]$ and any $x$, $c^i_x$ is no more than $c_x|T_i| /\sum |T_i|+1$ and no less than $c_x|T_i| /\sum |T_i|-1$ for any $i \in [N]$ and every $x$. The existence of this partition is ensured by Lemma \ref{bili}. Note that for any $i\in [N]$ the number of odd cycles in $\overline{\mathcal{F}}_i$ is no more than $k'|T_i| /\sum |T_i|+\sum_x 1\leq k'|T_i| /\sum |T_i|+3N.$
Let $A=\{a_1,a_2,\dots\}\subseteq \{6,8,10,12,14\}$ denote the $x$ such that $c_x\geq n\eta^3 /2x$. If $A$ is $\{8,12\}$ or any set with order $1$, then the problem could be solved by Proposition \ref{SCS}, Thus, we only need to consider the case where $A$ is neither $\{8,12\}$ nor any set with order $1$.  In this case, the greatest common divisor of $a_1,a_2,\dots$ is $2$, therefore there exist integers $b_1,b_2,\dots$ such that $a_1b_1+a_2b_2+\dots=2$.

Define a \emph{A-transformation} from a set of cycles $\mathcal{C''}$ to another set of cycles $\mathcal{C'}$: $\mathcal{A}:\mathcal{C''}\rightarrow \mathcal{C'}$ as following: for every $a_l\in A$, if $b_l$ is positive, then remove $b_l$ numbers of $C_{a_l}$ from $\mathcal{C''}$ into $\mathcal{C'}$;  if $b_l$ is negative, then remove $-b_l$ numbers of $C_{a_l}$ from $\mathcal{C'}$ into $\mathcal{C''}$. After a A-transformation from $\mathcal{C''}$ to $\mathcal{C'}$,  $\sum_{C \in \mathcal{C'}} |C|$ increases by 2 compared to before while $\sum_{C \in \mathcal{C''}} |C|$ decreases by 2 compared to before. 

Note that for every $i$ and every $x$, $\big|x c_x^i-x c_x|T_i| /\sum |T_i|\big|\leq x$, thus we have: 
\[
\big|\sum_{C \in \overline{\mathcal{F}}_i} |C|-|T_i|\big|\leq \sum_x \big|x c_x^i-x c_x|T_i| /\sum |T_i|\big|
\leq 6+8+10+12+14+2(N-1)\log n \leq 2N \log n.
\]

We use the following cyclic process for each $i \in [N-1]$:

In step $i$, do some A-transformation from $\overline{\mathcal{F}}_{i}$ to $\overline{\mathcal{F}}_{i+1}$ if $\sum_{C \in \overline{\mathcal{F}}_i} |C|-|T_i|>0$ or do some A-transformation from $\overline{\mathcal{F}}_{i+1}$ to $\overline{\mathcal{F}}_i$ if $\sum_{C \in \overline{\mathcal{F}}_i} |C|-|T_i|<0$ until $\sum_{C \in \overline{\mathcal{F}}_i} |C|-|T_i|=0$. Since $\big|\sum_{C \in \overline{\mathcal{F}}_i} |C|-|T_i|\big|\leq i 2N \log n$, by no more than $i N \log n$ numbers of A-transformation this progress could be reached. Since every $x \in A$ satisfies $c_x\geq n\eta^3 /2x$ and thus for every $x\in A$ and every $j\in [N]$, $c^j_x\gg i N \log n(|b_1|+|b_2|+\cdots)$, these A-transformation exists.

Denote these cycles as $\mathcal{F}_i$. Note that the number of odd cycles in  $\mathcal{F}_i$ is no more than $k'|T_i| / \sum |T_i| + 3N + (i-1) N \log n(|b_1|+|b_2|+\cdots)+i N \log n(|b_1|+|b_2|+\cdots)  \leq \frac{k (1 + \eta^{3})}{N}$.

Having done this step, we have $$\big|\sum_{C \in \overline{\mathcal{F}}_{i+1}} |C|-|T_{i+1}|\big|\leq 2N\log n+i 2N\log n=(i+1) 2N\log n,$$ and the number of odd cycles in $\overline{\mathcal{F}}_{i+1}$ is no more than $k'|T_i| / \sum |T_i| + 3N+i N\log n(|b_1|+|b_2|+\cdots).$ This implies the following fact.

\begin{fact}\label{processfact2}
    Having finished step $i$ we have $\big|\sum_{C \in \overline{\mathcal{F}}_{i+1}} |C|-|T_{i+1}|\big|\leq (i+1) 2N \log n$ and the number of odd cycles in $\overline{\mathcal{F}}_{i+1}$ is no more than $k'|T_{i+1}| / \sum |T_{i+1}|+3N+i N \log n(|b_1|+|b_2|+\cdots)$.
\end{fact}

When it comes to $\mathcal{F}_N$, since every other triple is used up, its order is the sum of the number of the cycles left in $\mathcal{C}_0$. And the number of odd cycles in it is no more than $k'|T_i| / \sum |T_i| + 3N+N^2\log n(|b_1|+|b_2|+\cdots)\leq \frac{k (1 + \eta^{3})}{N}$.

\end{proof}

\subsubsection{Step 3}
Now, for each triple still denoted as $T_i = (X_i, Y_i, Z_i)$ for every $i$, since no more than $N^{10} \log n$ vertices were used during the order adjustment to balance the triples and embed the cycles in $\mathcal{C}_2$, it remains $(2\varepsilon^*, d^*/2)$-half-superregular by Lemma \ref{slicinglemma}. 

Moreover, for every $i\in [N]$, we have:
$$
\frac{(2a+b)}{a} \times \min\{|X_i|,|Y_i|\} - (|X_i|+|Y_i|) - |Z_i| \geq \sqrt{m} - (a+b) N^{10} \log n > 0.
$$
Additionally, since
$$
\frac{|Z_i|+\sqrt{m}}{|X_i|+|Y_i|+2N^{10}\log n} \leq \frac{b}{2a},
$$
it follows that
$$
\frac{|Z_i|}{|X_i|+|Y_i|} < \frac{|Z_i|+\sqrt{m}}{|X_i|+|Y_i|+2N^{10}\log n} \leq \frac{b}{2a}.
$$
Furthermore, we have
$$
|Z_i| - \max\{|X_i|,|Y_i|\} \geq \eta^3 m - N^{10} \log n > 0.
$$
Thus, if we denote $(X_i, Y_i, Z_i)$ as $(V_1, V_2, V_3)$ with $|V_1| \leq |V_2| \leq |V_3|$, we conclude that
$$
\begin{aligned}
    & \frac{|V_1|+|V_2|}{|V_3|} \geq \frac{2a}{b}, \\
    & \frac{a+b}{a} |V_1| \geq |V_2| + |V_3|.
\end{aligned}
$$

\begin{claim}\label{endembed}
 For every $i \in [N]$, there is a perfect embedding from $\mathcal{F}_i$ into $T_i$ by Lemma \ref{weakblowup}.
\end{claim}

\begin{proof}
We prove it for each $i \in [N]$. Denote $|T_i| = m^*$ and let the number of odd cycles in $\mathcal{F}_i$ be $k'_i$. It follows that $0 <  \varepsilon^* \ll d^* ,\xi\ll \eta \ll 1$ and $T_i$ is $(2\varepsilon^*, d^*/2)$-half-superregular with $m^* = \sum_{C \in \mathcal{F}_i}|C|$. Besides, $\frac{1}{m^*}<\frac{a+b}{m}\ll\varepsilon^*$.

Since $k'_i \leq \frac{k(1+\eta^3)}{N}+2\sqrt{n}$ and $\frac{(1-3\sqrt{\xi})n}{N} \leq m^* \leq \frac{n}{N}$, we have 
$$
\frac{2m^* - 4k'_i}{m^* + 2k'_i} \geq \frac{2n (1-3\sqrt{\xi})-4k(1+\eta^3)-8\sqrt{n}}{(n + 2k)(1+\eta^3)+4\sqrt{n}} \geq \frac{2n - 4k}{(n + 2k)(1+\eta^{2.5})}.
$$
Since $a, b$ are $(n,k,\eta)$-good, we have $b/a\leq 2-\eta^2$ and
$$
\frac{b}{a} \leq \frac{2n - 4k}{n + 2k} - \eta^2 \leq (1+\eta^{2.5})\frac{2m^* - 4k'_i}{m^* + 2k'_i} - \eta^2 \leq \frac{2m^* - 4k'_i}{m^* + 2k'_i} - 0.1\eta^2.
$$
Thus, we obtain
$$
\frac{m^* a}{2a+b} = \frac{m^*}{2+b/a} \geq m^*/4 + k'/2.
$$

Let $2x_i = |V_j| - x_j + |V_h| - x_h$ for every $\{i,j,h\} = \{1,2,3\}$. We have
$$
x_1 = \frac{1}{2}(|V_2| + |V_3| - |V_1|), \quad x_2 = \frac{1}{2}(|V_1| + |V_3| - |V_2|), \quad x_3 = \frac{1}{2}(|V_1| + |V_2| - |V_3|).
$$ 

We claim that there exists a tripartite graph $G_0$ with each part of order $x_1, x_2, x_3$, consisting of vertex-disjoint cycles $C_{m_1/2}, C_{m_2/2}, \ldots, C_{m_r/2}$ where $k'_i$ of these are odd. Let $\{i, j, h\} = \{1, 2, 3\}$ and $x_h = \max\{x_i, x_j, x_h\}$. Since $\frac{a+b}{a}|V_1| \geq |V_2| + |V_3|=m^*-|V_1|$, we have $|V_1| \geq \frac{m^* a}{2a+b}$. Thus the minimum degree of the complete graph $K_{x_i x_j x_h}$ is no less than
$$
\begin{aligned}
   \min \{x_i+x_j, x_j+x_h, x_i+x_h\} = |V_1| \geq \frac{m^* a}{2a+b} \geq m^*/4 + k'_i/2 = \frac{x_1+x_2+x_3}{2} + k'_i/2.
\end{aligned}
$$
By Lemma \ref{step3graph}, there exists a $3$-partite graph $G_0$ with each part of order $x_1, x_2, x_3$, consisting of vertex-disjoint cycles $C_{\frac{m_1}{2}}, C_{\frac{m_2}{2}}, \ldots, C_{\frac{m_r}{2}}$ where $k'_i$ of these are odd.

Taking a uniformly random partition of $V_1, V_2, V_3$ into disjoint sets such that $V_1 = X_1 \cup X_2$, $V_2 = Y_1 \cup Y_2$, $V_3 = Z_1 \cup Z_2$ with $|X_1| = x_1$, $|Y_1| = x_2$ and $|Z_1| = x_3$. For any $i \in [3]$, the number of edges in $G_0$ adjacent to the $x_i$ part is $2x_i$, and for any $i, j, h \in [3]$, the number of edges between the $x_i$ part and the $x_j$ part is $x_i + x_j - x_h$, and $(x_i + x_j - x_h)+x_h=|V_h|$. Since $b/a\leq 2-\eta^2$ and $\frac{a+b}{a}|V_1| \geq |V_2| + |V_3|$, 
$$
\begin{aligned}
    3|V_1|-|V_2|-|V_3|\geq 3|V_1|-\frac{a+b}{a}|V_1|\geq m^*\frac{a}{2a+b}(3-\frac{a+b}{a}) =m^*\frac{2a-b}{2a+b}\geq \frac{\eta^2m^*}{4},
\end{aligned}
$$
we have $|V_i| - x_i = 1.5|V_i| - 0.5(|V_j| + |V_h|) \geq \frac{\eta^2m^*}{8}$, thus $\min \{|X_2|,|Y_2|,|Z_2|\} \geq \frac{\eta^2m^*}{8}$. On another aspect, since $\frac{|V_1| + |V_2|}{|V_3|} \geq 2a/b$, we have
$$
\begin{aligned}
   \min \{|X_1|,|Y_1|,|Z_1|\} \geq \frac{1}{2}(|V_1| + |V_2| - |V_3|) \geq \frac{1}{2}m^*\frac{2a-b}{2a+b}\geq\frac{\eta^2m^*}{8},
\end{aligned}
$$
thus $\min \{|X_1|,|Y_1|,|Z_1|\} \geq \frac{\eta^2m^*}{8}$.

By Lemma \ref{slicinglemma}, we can assert that, with high probability, each of the triples $(X_2, Y_1, Z_2)_\mathcal{H'}$, $(X_2, Z_1, Y_2)_\mathcal{H'}$, and $(Z_2, X_1, Y_2)_\mathcal{H'}$ is $(\frac{8}{\eta^2}\varepsilon^*, {d^*}^2 / 64)$-half-superregular, thus by union bound there exists such a partition such that all of the triples $(X_2, Y_1, Z_2)_\mathcal{H'}$, $(X_2, Z_1, Y_2)_\mathcal{H'}$, and $(Z_2, X_1, Y_2)_\mathcal{H'}$ are $(\frac{8}{\eta^2}\varepsilon^*, {d^*}^2 / 64)$-half-superregular, and let $G_0=G_0(X_1,Y_1,Z_1)$. Applying Theorem \ref{weakblowup} with $(m=m^*, \varepsilon=\frac{8}{\eta^2}\varepsilon^*, d={d^*}^2 / 64, \delta=\frac{\eta^2}{8}, \Delta=2, r=3)$, we can embed a 1-expansion of $G_0$, which contains all of the loose cycles in $\mathcal{F}_i$.

\end{proof}
\section{Concluding remarks}\label{concluding}
\begin{itemize}
    \item In this paper, we applied the regularity method and the blow-up lemma to derive an asymptotically tight bound for the embedding problem of disjoint loose cycles in 3-graphs. Based on our results, we propose the following conjecture concerning the tight bounds for embedding disjoint loose cycles:

    \begin{conjecture}
    There exists an integer $n_0$ such that for all $n \geq n_0$, the following holds: Let $\mathcal{C}$ be a $3$-graph consisting of vertex-disjoint loose cycles $C_{n_1}, C_{n_2}, \dots, C_{n_r}$ such that $\sum_{i=1}^{r} n_i = n$, and let $k$ denote the number of loose cycles with odd lengths. If $\mathcal{H}$ is an $n$-vertex $3$-graph with $\delta_2(\mathcal{H}) \geq \frac{n + 2k}{4}$, then $\mathcal{H}$ contains $\mathcal{C}$ as a spanning subhypergraph.
    \end{conjecture}

    \item Additionally, we pose the following general problem regarding the embedding of disjoint loose cycles in large $k$-graphs:

    \begin{question}
    Let $k > 3$ be a positive integer, and let $\mathcal{C}$ be a $k$-graph consisting of vertex-disjoint loose cycles $C_{n_1}, C_{n_2}, \dots, C_{n_r}$ such that $\sum_{i=1}^{r} n_i = n$. Let $\mathcal{H}$ be an $n$-vertex $k$-graph. What is the minimum codegree condition on $\mathcal{H}$ that guarantees $\mathcal{C}$ is a spanning subhypergraph of $\mathcal{H}$ when $n$ is sufficiently large?
    \end{question}

    \item Finally, we examine the $1$-expansion version of the Hajnal-Szemerédi theorem, which leads to the following related question:

    \begin{question}
    Let $t$ be a positive integer, and let $K^+_t$ denote the $1$-expansion of $K_t$. Given an $n$-vertex $3$-graph $\mathcal{H}$, what is the minimum codegree condition on $\mathcal{H}$ that ensures the existence of a $K^+_t$-factor when $n$ is sufficiently large, and when $t + \frac{t(t-1)}{2}$ divides $n$?
    \end{question}
\end{itemize}

\section{Acknowledgements}
We are sincerely grateful to Professor Jie Han for his insightful suggestions. Yangyang Cheng was supported by the PhD studentship of ERC Advanced Grant (883810). Guanghui Wang was supported by the Natural Science Foundation of China (12231018).

\appendix
\makeatletter
\renewcommand{\thecase}{\arabic{case}}
\@addtoreset{case}{section}  
\makeatother
\section{Appendix}

\subsection{}\label{appendix}
In this subsection, we deduce Theorem \ref{oddandeven} from Theorem \ref{reducethm0}. The proof is similar with the work of Khan in \cite{khan2011spanning}.
First, we prove that every $n$-vertex $3$-graph with minimum codegree at least $(1/4+o(1))n$ contains a loose cycle on any even number of vertices as shown in Lemma \ref{pancyclic}. For this, we need the following two results.

\begin{lemma}[\cite{chva}]\label{pro}
Let
\[F(n, \lambda, k, \delta)=\binom{n}{k}^{-1}\sum_{i=0}^{(\lambda-\delta)k}\binom{\lambda n}{i}\binom{(1-\lambda)n}{k-i}\]
denote the probability that a random subset of $k$ elements out of $n$, $\lambda n$ of which are marked, contains at most $(\lambda-\delta)k$ marked elements; then \[F(n, \lambda, k, \delta)\leq e^{-2\delta^2k}.\]
\end{lemma}
\begin{theorem}[\cite{czygrinow2014siam}]\label{2014siam}
There is an integer $n_0$ such that every $3$-graph $\mathcal{H}$ on $n\geq n_0$ vertices where $n$ is even and with minimum codegree at least $n/4$ contains a loose Hamilton cycle.
\end{theorem}
\begin{lemma}\label{pancyclic}
For any $\gamma>0$, there exists a constant $N=N(\gamma)$ such that if $\mathcal{H}$ is a $3$-graph on $n\geq N$ vertices with minimum codegree $\delta_2(\mathcal{H})\geq (1/4+\gamma)n$, then $\mathcal{H}$ is super-pancyclic, i.e., for any even integer $q$ with $6\leq q\leq n$, $\mathcal{H}$ contains a loose cycle on $q$ vertices.
\end{lemma}
\begin{proof}
The proof is divided into two cases depending on $q$.

For $6\leq q\leq n/4$, choose a loose path $P=v_1v_2\ldots v_{q-1}$ on $(q-1)$ vertices (such a path exists as $\mathcal{H}$ contains a loose Hamilton cycle by Theorem \ref{ko}). Since $\delta_2(\mathcal{H})\geq (1/4+\gamma)n$, we have $|N_{V(\mathcal{H})\backslash V(P)}(v_1, v_{q-1})|\geq\gamma n>0$. Then, we can choose a vertex $v_q\in N_{V(\mathcal{H})\setminus V(P)}(v_1, v_{q-1})$ such that $v_1v_2\ldots v_{q-1}v_q$ is a loose cycle on $q$ vertices, as desired.

For $q>n/4$, choose a subset $A\subseteq V(\mathcal{H})$ with $|A|=q$ among all such sets uniformly at random. For $u,v\in V(\mathcal{H})$, let $X_{u,v}$ be the random variable $|N_{\mathcal{H}}(u,v)\cap A|$. Then, by Lemma \ref{pro}, we have
\[\mathds{P}(X_{u,v}\leq q/4)\leq e^{-2\gamma^2q}\leq e^{-\gamma^2n/2}.\]
By the union bound,
\[\mathds{P}(\exists u, v\in V(\mathcal{H}): X_{u, v}\leq q/4)\leq n^2\max_{u, v\in V(\mathcal{H})}\mathds{P}(X_{u, v}\leq q/4)\leq n^2e^{-\gamma^2n/2}.\]
We can choose $n'$ such that for $n\geq n'$, this probability tends to 0. This implies that there exists a subset $A\subseteq V(\mathcal{H})$ with $|A|=q>n/4$ such that for each pair $u,v\in A$, we have $X_{u,v}\geq |A|/4$. Choose $n_0$ to be an integer such that Theorem \ref{2014siam} holds. Let $N=\max\{n', n_0\}$. Therefore, there is a loose Hamilton cycle in $\mathcal{H}[A]$ by Theorem \ref{2014siam}.
\end{proof}

Recall that $C_t$ denotes a loose cycle on $t$ vertices. Suppose that $\mathcal{H}$ is an $n$-vertex $3$-graph with $\delta_2(\mathcal{H})\geq(n+2k)/4+\eta n\geq(1/4+\eta)n$, and $\mathcal{C}$ is a $3$-graph consisting of vertex-disjoint union of $C_{n_1}, C_{n_2}, \ldots, C_{n_r}$ satisfying $\sum_{i=1}^{r}n_i=n$. 
Let $\mathcal{C}_1$ be the set of loose cycles on at least $\log n$ vertices in $\mathcal{C}$ and $\mathcal{C}_2=\mathcal{C}\backslash \mathcal{C}_1$. 
Suppose that $\mathcal{C}_1$ is not an empty set. Define $|V(\mathcal{C}_1)|=cn$. We consider the following cases.

\begin{case}
$0< c<\eta/2.$
\end{case}
Using Lemma \ref{pancyclic}, for given $\eta>0$, we can choose a suitable integer $N$ such that
the loose cycles in $\mathcal{C}_1$ are embedded into $\mathcal{H}$ one by one. Indeed, let $\mathcal{H}'=\mathcal{H}[V(\mathcal{H})\setminus V(\mathcal{C}_1)]$, we have $\delta_2(\mathcal{H}')\geq(1/4+\eta)n-\eta n/2\geq(1/4+\eta/2)|\mathcal{H}'|$. We can embed the loose cycles of $\mathcal{C}_2$ into $\mathcal{H}'$ as Theorem \ref{reducethm0}.

\begin{case}
$\eta/2\leq c<1-\eta/2.$
\end{case}

In this case, we randomly partition $V(\mathcal{H})$ into $V_1$ and $V_2$ with $|V_1|=|V(\mathcal{C}_1)|$ and $|V_2|=|V(\mathcal{C}_2)|$. For $u,v\in V(\mathcal{H})$, let $X_{u,v}^{i}$ be the random variable $|N_{\mathcal{H}}(u,v)\cap V_i|$, where $i\in[2]$. By Lemma \ref{pro}, we have
$$\mathds{P}(X_{u,v}^i\leq (1/4+\eta/2)|V_i|)\leq e^{-\eta^2|V_i|/2}.$$
By the union bound,
$$\begin{aligned}
\mathds{P}(\exists u, v\in V_i: X_{u, v}^i\leq (1/4+\eta/2)|V_i|)&\leq |V_i|^2\max_{u, v\in V_i}\mathds{P}(X_{u, v}^i\leq (1/4+\eta/2)|V_i|)\\
&\leq |V_i|^2e^{-\eta^2|V_i|/2}.
\end{aligned}$$
Therefore, we can choose $n_0$ such that for $n\geq n_0$, this probability tends to 0. This implies that
there is a partition $V(\mathcal{H})=V_1\cup V_2$ such that $\delta_2(\mathcal{H}[V_1]) \geq(1/4+\eta/2)|V_1|$ and $\delta_2(\mathcal{H}[V_2]) \geq(1/4+\eta/2)|V_2|$. Then, the loose cycles of $\mathcal{C}_1$ can be embedded into $\mathcal{H}[V_1]$ as Case 4, and the loose cycles of $\mathcal{C}_2$ can be embedded into $\mathcal{H}[V_2]$ by Theorem \ref{reducethm0}.

\begin{case}
$1-\eta/2\leq c<1.$
\end{case}

Note that $0<|V(\mathcal{C}_2)|\leq\eta n/2$ in this case. We can embed loose cycles in $\mathcal{C}_2$ one by one into $\mathcal{H}$ by Lemma \ref{pancyclic}. Let $\mathcal{H}'=\mathcal{H}[V(\mathcal{H})\setminus V(\mathcal{C}_2)]$, we know that $\delta_2(\mathcal{H}')\geq(1/4+\eta/2)n\geq(1/4+\eta/2)|\mathcal{H}'|$. Then the loose cycles of $\mathcal{C}_1$ can be embedded into $\mathcal{H}'$ as Case 4.

\begin{case}
$c=1$, i.e., all loose cycles in $\mathcal{C}$ has larger than $\log n$ vertices.
\end{case}

Assume that $n_1\geq n_2\geq\ldots\geq n_r>\log n$. We consider two cases relying on $n_1$.

\begin{subcase}
$n_1\geq(1-\eta/2)n.$    
\end{subcase}

We embed all loose cycles except for $C_{n_1}$ one by one using Lemma \ref{pancyclic}. Indeed, let $\mathcal{H}'=\mathcal{H}[V(\mathcal{H})\setminus \bigcup_{i=2}^{r}V(C_{n_i})]$, we know that $\delta_2(\mathcal{H}')\geq(1/4+\eta/2)n\geq(1/4+\eta/2)|\mathcal{H}'|$. Then $\mathcal{H}'$ contains a loose Hamilton cycle which means that $C_{n_1}$ can be embedded.

\begin{subcase} 
$n_i<(1-\eta/2)n$ for each $i\in[r].$
\end{subcase}

Divide the loose cycles of $\mathcal{C}$ into two parts $\mathcal{C}_1^1$ and $\mathcal{C}_1^2$ such that $|V(\mathcal{C}_1^1)|\leq(1-\eta/2)n$ and $|V(\mathcal{C}_1^2)|\leq(1-\eta/2)n$. Note that $|V(\mathcal{C}_1^j)|\geq\eta n/2$ for each $j\in[2]$. Then we randomly partition $V(\mathcal{H})$ into two sets $V_1^1$ and $V_1^2$ such that $|V_1^1|=|V(\mathcal{C}_1^1)|$ and $|V_1^2|=|V(\mathcal{C}_1^2)|$. In the same way as Case 2, we know that there is a vertex partition $V(\mathcal{H})=V_1^1\cup V_1^2$ such that $\delta_2(\mathcal{H}[V_1^1])\geq(1/4+\eta/2)|V_1^1|$ and $\delta_2(\mathcal{H}[V_1^2])\geq(1/4+\eta/2)|V_1^2|$. 
In the following, we will embed the loose cycles of $\mathcal{C}_1^j$ into $\mathcal{H}[V_1^j]$ for $j\in[2]$. If $\mathcal{C}_1^j$ contains a loose cycle on at least $(1-\eta/2)|V_1^j|$ vertices for some $j\in[2]$, then $\mathcal{C}_1^j$ can be embedded into $\mathcal{H}[V_1^j]$ as Subcase 4.1. Otherwise, recursively apply the same partitioning process until the condition of Subcase 4.1 is satisfied.
Let $k$ denote the times of the vertex partitions. Then there are $2^k$ vertex subsets.
For $i\in[k]$ and $j\in[2^k]$, we use $V_{i}^{j}$ and $\mathcal{C}_{i}^{j}$ to denote the vertex subsets and the sets of loose cycles.
Since the partitioning process terminates before the number of vertices in the subhypergraph is less than $\log n$, we have $k<\log n$. 
Note that $\eta^kn/2^k\leq|V_{k}^{j}|\leq (1-\eta/2)^{k}n$ for each $j\in[2^k]$.
By Lemma \ref{pro} and the union bound, we have

$$\begin{aligned}
&\mathds{P}\left(\exists u, v\in V_k^j: |N_{\mathcal{H}}(u,v)\cap V_{k}^j|\leq (1/4+\eta-\sum_{t=1}^{k}(\eta/2)^t)|V_k^j|\right)\\
&\leq |V_k^j|^2\max_{u, v\in V_k^j}\mathds{P}\left(|N_{\mathcal{H}}(u,v)\cap V_{k}^j|\leq (1/4+\eta/4)|V_k^j|\right)\\
&\leq |V_k^j|^2e^{-9\eta^2|V_k^j|/8}\longrightarrow 0,
\end{aligned}$$
where $j\in[2^k]$.
This means that the minimum codegree condition required in Lemma \ref{pancyclic} remains satisfied in the subhypergraph of $\mathcal{H}$ until the partitioning process terminates.
Hence, the minimum codegree conditions are satisfied in each partitioning process. 
If Subcase 4.1 satisfies, we can embed loose cyclses in the subhypergraph. 
Otherwise, we can apply Lemma \ref{pancyclic} to find a loose Hamilton cycle in the subhypergraph. We finish the proof.



\subsection{}

\begin{lemma}\label{bili}
    For any positive integer $q$ and any sequence of positive numbers $a_1, a_2, \ldots, a_k$ satisfying $\sum_{i=1}^k a_i=1$, there exists a sequence of positive integers $q_1,q_2,\ldots ,q_k$ with $\sum_{i=1}^k q_i=q$ such that $|q_i-qa_i|\leq 1$ for each $i\in [k]$.
\end{lemma}

\begin{proof}
    Let $w=q-\sum_{i=1}^k \lfloor qa_i \rfloor$, we have
    $w\leq q-\sum_{i=1}^k (qa_i-1)\leq k$. 
    Let $q_i=\lfloor qa_i \rfloor+1$ for each $i\in [w]$ and let $q_i=\lfloor qa_i \rfloor$ for each $i\in\{w, w+1, \ldots, k\}$. 
    Thus, we have $\sum_{i=1}^k q_i=w+\sum_{i=1}^k \lfloor qa_i \rfloor=q$ and for each $i\in [k]$, we have
    $$|q_i-qa_i|\leq \max \{qa_i-\lfloor qa_i \rfloor, \lceil qa_i \rceil-qa_i \}\leq 1.$$
\end{proof}



\end{document}